\documentclass[12pt]{article}

\usepackage{amssymb,amsmath}
\usepackage{indentfirst}
\usepackage[english]{babel}
\usepackage{amsfonts}
\usepackage{amsthm}
\usepackage{graphicx}
\usepackage{epsfig}
\usepackage{dsfont}
\usepackage[all]{xy}

\newtheorem{theorem}{Theorem}

\newtheorem{lemma}[theorem]{Lemma}
\newtheorem{corollary}[theorem]{Corollary}
\newtheorem{remark}[theorem]{Remark}

\newtheorem{definition}[theorem]{Definition}

\newcommand{\R}{\mathbb{R}}
\newcommand{\Ee}{\mathbb {E}}
\newcommand{\Q}{\mathbb{Q}}

\newcommand{\Sf}{\mathbb{S}}

\newcommand{\Hy}{\mathbb{H}}

\newcommand{\hess}{\mbox{Hess\,}}
\newcommand{\grad}{\mbox{grad\,}}

\def\<{\langle}

\def\>{\rangle}

\def\va{\varphi}

\def\e{\epsilon}
\def\d{\partial}
\def\bea{\begin{eqnarray*} }
\def\eea{\end{eqnarray*} }
\def\be{\begin{equation} }
\def\ee{\end{equation} }

\usepackage{xcolor}

\begin{document}

\title{Geometry of submanifolds with respect to ambient vector fields}

\author{Fernando Manfio, Ruy Tojeiro and Joeri Van der Veken
\footnote{Parts of this work were carried out while J. Van der Veken 
visited Universidade de S\~{a}o Paulo in the framework of the 
Young Researchers Summer Program 2018, and while Ruy Tojeiro
visited the KULeuven partially supported by Aucani (USP) Public
Notice 967/2018. J. Van der Veken is 
supported by project 3E160361 of the KU Leuven Research Fund 
and by EOS project G0H4518N of the Belgian government. 
Ruy Tojeiro is supported by Fapesp grant 2016/23746-6 and 
CNPq grant 303002/2017-4.
}}

\date{}
\maketitle

\vspace{2ex}

\noindent \emph{Key words and phrases:} {\small {\em Constant ratio property,  Principal direction property, \\
Radial vector field, Killing vector field, conformal Killing vector field, loxodromic \\isometric immersion.}}

\begin{abstract}
\vspace{1ex}

Given a Riemannian manifold $N^n$ and ${\cal Z}\in \mathfrak{X}(N)$, an isometric 
immersion $f\colon M^m\to N^n$ is said to have the \emph{constant ratio property 
with respect to ${\cal Z}$} either if the tangent component ${\cal Z}^T_f$ of 
${\cal Z}$ vanishes identically  or if ${\cal Z}^T_f$ vanishes nowhere and the ratio 
$\|{\cal Z}^\perp_f\|/\|{\cal Z}^T_f\|$ between the lengths of the normal and 
tangent components of ${\cal Z}$  is constant along $M^m$. 
It has  the \emph{principal direction property with respect to ${\cal Z}$} if 
${\cal Z}^T_f$ is an eigenvector of all shape operators of $f$ at all points of $M^m$. 
In this article we study isometric immersions $f\colon M^m\to N^n$ of arbitrary
codimension that have either the constant ratio or the  principal direction property
with respect to distinguished vector fields  ${\cal Z}$ on space forms, product spaces 
$\Sf^n\times \R$ and $\Hy^n\times \R$, where $\Sf^n$ and $\Hy^n$ are the 
$n$-dimensional sphere and hyperbolic space, respectively, and, more generally, 
on warped products $I\times_{\rho}\Q_\e^n$ of an open interval $I\subset \R$ 
and a space form $\Q_\e^n$. Starting from the observation that these properties
are invariant under conformal changes of the ambient metric, we provide new 
characterization and classification results 
of isometric immersions that satisfy either of those properties, 
or both of them simultaneously,
for several relevant instances of ${\cal Z}$ as well as simpler descriptions 
and proofs of some known ones for particular cases of  ${\cal Z}$ previously 
considered by many authors. Our methods also allow us to classify Euclidean 
submanifolds with the property that the normal components 
of their position vector fields are parallel with respect to the normal connection, 
and to give  alternative descriptions to those in \cite{byc3} of Euclidean
submanifolds whose tangent or normal components of their position vector fields 
have constant length. 
\end{abstract}

\section{Introduction}

   Let $M^m$ and $N^n$ be Riemannian manifolds of dimensions $m$ and $n$, respectively.
Given a vector field ${\cal Z}$ in $N^n$, it is a natural problem to investigate the 
isometric immersions $f\colon M^m\to N^n$ that have relevant geometric properties with 
respect to ${\cal Z}$.  
   
 For example, for which oriented hypersurfaces $f\colon M^m\to N^{m+1}$ does a unit 
normal vector field along $f$ make a constant angle with ${\cal Z}$? 
A related problem is to look for the hypersurfaces for which the tangent component of 
${\cal Z}$ is a principal direction of $f$ at any point of $M^m$.

    Equivalent versions of these problems can also be posed for isometric immersions 
$f\colon M^m\to N^n$ of arbitrary codimension $n-m$. Namely, the first one is equivalent 
to looking for the hypersurfaces for which either the tangent component of ${\cal Z}$ 
vanishes everywhere or it is nowhere vanishing and the ratio between the lengths of the 
normal and tangent components of ${\cal Z}$ is constant, and in this form the problem makes 
sense for submanifolds of arbitrary codimension. One may also consider the second of the 
preceding problems for submanifolds of any codimension by requiring the tangent component 
of ${\cal Z}$ to be a principal direction of \emph{all} shape operators of $f$. 
It will be convenient to refer to isometric immersions satisfying those conditions as 
isometric immersions having \emph{the constant ratio property} or 
\emph{the principal direction property}  with respect to ${\cal Z}$, respectively.
     
   These problems are particularly interesting when $N^n$ is endowed with 
distinguished vector fields ${\cal Z}$, for instance if $N^n$ is a space 
form or a Riemannian product $\Q_\e^n\times\R$, where $\Q_\e^n$ denotes either
the sphere $\Sf^n$ or the hyperbolic space $\Hy^n$ of dimension $n$,
depending on whether $\e=1$ or $\e=-1$, respectively. More generally, if 
$N^n$ is a warped product $I\times_{\rho}P$ of an open interval $I\subset \R$ 
with any Riemannian manifold $P$.
In particular, in Euclidean space one may consider the above problems
when ${\cal Z}$ is either a constant or the radial vector field, or when it 
is a Killing or a conformal Killing vector field. In a product space
$\Q_\e^n\times\R$ or, more generally, in any warped product $I\times_{\rho} P$,
a natural vector field ${\cal Z}$ to consider
is a unit  vector field $\frac{\partial}{\partial t}$ tangent to either the factor $\R$
or the factor $I$, respectively.
One may also consider those problems for the ``radial" vector fields in $\Q_\e^n$, that is,
the closed and conformal vector fields in those spaces.

Several authors have addressed particular instances of the above problems for special cases 
of the ambient space $N^n$ and of the vector field ${\cal Z}$. In particular, hypersurfaces
with the constant ratio  property with respect to the 
unit vector field ${\cal Z}=\frac{\partial}{\partial t}$ tangent to the factor $\R$ of a 
product space $\Sf^{n-1}\times \R$ and  $\Hy^{n-1}\times \R$, or with respect to a constant 
vector field in Euclidean space, were described in \cite{dfjv1}, \cite{dm} and 
\cite{to}, and submanifolds with the principal direction property with respect to those vector 
fields in \cite{dfj1}, \cite{dmn},  \cite{bt} and \cite{to}. 
The first of these problems was  studied
in \cite{dmvv} for surfaces in a warped product $I \times_{\rho} \R^2$ and for hypersurfaces 
of arbitrary dimension of any warped product $I\times_{\rho}P$ in \cite{gprh2}. 
The constant ratio  property with respect to the radial vector field in Euclidean space was 
investigated in \cite{m} for surfaces in $\R^3$, for Euclidean hypersurfaces of any 
dimension in \cite{byc1} and \cite{y}, and for submanifolds of arbitrary codimension in \cite{byc2}.  
Surfaces in $\R^3$ with the constant ratio property with respect to a Killing vector field  
were studied in \cite{mn}. 
The principal direction property with respect to
the radial vector field was investigated in \cite{mf} for surfaces in $\R^3$ and 
with respect to the vector field  ${\cal Z}=\frac{\partial}{\partial t}$ tangent to the factor $\R$
for hypersurfaces of a warped product $\R\times_{\rho}P$ in \cite{gprh}.
Constant mean curvature hypersurfaces of $\R\times_{\rho}P$ with this property
were characterized in \cite{sh}, yielding a  description of constant mean 
curvature hypersurfaces with the principal direction property with respect to 
closed and conformal vector fields in $\Q_\e^n$. 

 Our first main result is a description of the isometric immersions  $f\colon M^m\to \R^{n}$,
 for arbitrary values of the dimension $m$ and the codimension $n-m$, 
that have the constant ratio property with respect to a constant vector field, 
as well as of the isometric immersions  $f\colon M^m\to \Q_\epsilon^{n}\times \R$ that have the constant 
ratio property with respect to the unit vector field ${\cal Z}=\frac{\partial}{\partial t}$ tangent to the 
factor $\R$.

   Several of the remaining results of this article rely on the elementary but useful observation that, 
for an isometric immersion $f\colon M^m\to N^n$, both the principal direction and constant 
ratio properties with respect to a vector field ${\cal Z}$ on $N^n$ are invariant under conformal 
changes of the metric of the 
ambient space. Equivalently, if $\Psi\colon N^n\to \hat N^n$ is a (local) conformal diffeomorphism, 
then any of these properties holds for an isometric immersion $f\colon M^m\to N^n$  with respect to 
${\cal Z}\in \mathfrak{X}(N)$  if and only if it holds for $\hat f=\Psi\circ f$ with 
respect to $\hat{\cal Z} = \Psi_*{\cal Z}\in \mathfrak{X}(\hat N)$. 

By applying the preceding observation, we first extend the description of isometric immersions  
$f\colon M^m\to \Q_\epsilon^{n}\times \R$ with the constant ratio property with respect to 
the unit vector field ${\cal Z}=\frac{\partial}{\partial t}$ tangent to the factor $\R$ 
to the case in which the ambient space is any warped product $I\times_{\rho}\Q_\epsilon^{n}$, 
with $I\subset \R$ an open interval, making use of the fact that any warped product metric 
on $I\times\Q_\epsilon^{n}$ is conformal to the standard Riemannian product metric. 
We then use the well known representations of (open dense subsets of) $\Q_\epsilon^{n+1}$ 
as such a warped product to describe all ``loxodromic" isometric immersions into a space
form $\Q_\epsilon^{n+1}$, that is, isometric immersions into  $\Q_\epsilon^{n+1}$ that have the 
constant ratio property with respect to a ``radial" vector field in  $\Q_\epsilon^{n+1}$.  
    
   Then we apply the same basic observation to the well-known conformal diffeomorphism of  
$\Sf^{n-1}\times \R$ onto $\R^{n}\setminus \{0\}$, and give a similar description of all isometric 
immersions  $f\colon M^m\to \R^{n}\setminus \{0\}$ that have  the constant ratio property  with 
respect to the radial  vector field, making the classification in \cite{byc2} more explicit.  
In the hypersurface case we recover the main result of \cite{y} (see also \cite{byc1} and \cite{m}).

We use the same idea to classify isometric immersions  $f\colon M^m\to \R^{n}\setminus \{0\}$ 
that have  the principal direction  property  with respect to the radial  vector field, extending
for arbitrary values of $m$ and $n$ the main result of \cite{mf}, where this problem was studied 
for surfaces in $\R^3$.
We also determine the  isometric immersions $f\colon M^m\to \R^{n}\setminus \{0\}$ that have both 
the constant ratio and principal direction properties with respect to the radial  vector field.

Our methods also allow us to classify the related class of Euclidean submanifolds with the property
that the normal components of their position vector fields are parallel with respect to the normal
connection, and  to give an alternative description of Euclidean submanifolds whose tangent
or normal components of their position vector fields have constant length. The latter were first
described in \cite{byc3}, where they were called $T$-constant and $N$-constant submanifolds.

    Our description of isometric immersions  $f\colon M^m\to \R^{n}\setminus \{0\}$ that have 
the constant ratio property with respect to the radial vector field is thus based on the observation 
that any such isometric immersion is the composition $f=\Psi\circ \hat f$ of an isometric immersion 
$\hat f\colon M^m\to \Sf^{n-1}\times \R$ that has the constant ratio  property with respect to 
$\frac{\partial}{\partial t}$ with the conformal diffeomorphism 
$\Psi \colon \Sf^{n-1}\times \R\to \R^{n}\setminus \{0\}$.  
Notice that in the very particular case in which $m=1$ and $n=2$, that is, unit speed curves in $\R^2$, 
our approach reduces to the observation that the logarithmic spiral, the only unit speed plane curve 
that has the constant ratio property with respect to the radial vector field (with nonvanishing tangent 
and normal components of that vector field), is just the image under the conformal diffeomorphism 
of $\Sf^{1}\times \R$ onto $\R^{2}\setminus \{0\}$ of a circular helix in $\Sf^{1}\times \R$.
   
   The corresponding conformal covering map of  $\Hy^{n-1}\times \R$ onto $\R^{n}\setminus \R^{n-1}$ 
is then used to describe all isometric immersions  $f\colon M^m\to \R^{n}\setminus \{0\}$ that have 
either the constant ratio or the principal direction property with respect to a Killing vector field 
${\cal K}$ in $\R^n$ generating rotations around a subspace $\R^{n-1}$ of $\R^{n+1}$.
The first of these problems was investigated in \cite{mn} for surfaces in $\R^3$. However, their
classification misses a large class of examples (see Remark \ref{munteanu} below).

  As one further application of our basic observation, we describe all isometric immersions into Euclidean 
space that have the constant ratio property or the principal direction property with respect to any of the
vector fields that generate the Lie algebra of \emph{conformal} Killing vector fields in $\R^n$.

\section{Some basic lemmas}
  
  Given an isometric immersion $f\colon M^m\to N^n$ and  ${\cal Z}\in \mathfrak{X}(N)$, we  always 
denote by ${\cal Z}^T_f\in \mathfrak{X}(M)$ and ${\cal Z}^\perp_f\in \Gamma(N_fM)$ the tangent and normal 
vector fields, respectively, given by orthogonally decomposing ${\cal Z}$ into its tangent and normal 
components along $f$, i.e.,
  $$
  {\cal Z}(f(x))=f_*(x){\cal Z}^T_f(x)+{\cal Z}^\perp_f(x)\;\;\;\mbox{for all $x\in M^m$}.
  $$

\begin{definition} \label{def:cr}
\emph{Given a Riemannian manifold $N^n$ and ${\cal Z}\in \mathfrak{X}(N)$, we say that an isometric immersion 
$f\colon M^m\to N^n$ has
\begin{itemize}
\item[(i)] the \emph{constant ratio property with respect to ${\cal Z}$} either if ${\cal Z}^T_f$ 
vanishes identically  or if ${\cal Z}^T_f$ vanishes nowhere and $\|{\cal Z}^\perp_f\|/\|{\cal Z}^T_f\|$ is 
constant along $M^m$;
\item[(ii)] the \emph{principal direction property with respect to ${\cal Z}$} if  ${\cal Z}^T_f$ 
is an eigenvector of all shape operators of $f$ at all points of $M^m$.
\end{itemize}}
\end{definition}

   Clearly, an isometric immersion $f\colon M^m\to N^n$ has either the constant ratio or the principal direction 
property with respect to ${\cal Z}\in \mathfrak{X}(N)$ if and only if it has the same property with 
respect to $\lambda\mathcal Z$ for some (and hence for any) $\lambda \in C^{\infty}(N)$ with no zeroes along $f(M)$.

 In case ${\cal Z} \in \mathfrak{X}(N)$ is a \emph{parallel} vector field in $N^n$,  the next lemma gives 
equivalent conditions for an isometric immersion $f\colon M^m \to N^n$ to have the principal direction or 
the constant ratio property with respect to $\mathcal Z$, and in particular clarifies how these properties
are related in this case. We denote by $\alpha_f$ the second fundamental form of $f$, by $A^f_{\xi}$ its
shape operator with respect to $\xi\in \Gamma(N_fM)$, and by $\nabla^{\perp}$ its normal connection. 
Given a Riemannian manifold $M^m$ and a nowhere vanishing $X\in \mathfrak{X}(M)$, we denote by $\{X\}^\perp$ 
the distribution of codimension one in $M^m$ given by the orthogonal complements of $X(x)$ in each tangent 
space $T_xM$, $x\in M^m$.

\begin{lemma}\label{prop:paralellvf}
Let $N^n$ be a Riemannian manifold that admits a parallel vector 
field ${\cal Z}$ and let  $f\colon M^m \to N^n$ be an isometric immersion.  
Then  $f$ has the principal direction property with respect  to 
${\cal Z}$ if and only if ${\cal Z}^T_f$ is nowhere vanishing and ${\cal Z}^\perp_f$ is parallel in the normal 
connection along $\{{\cal Z}^T_f\}^\perp$. If $n=m+1$, this is 
equivalent to $\|{\cal Z}^T_f\|$ being a nonzero constant along 
$\{{\cal Z}^T_f\}^\perp$. Moreover, the following assertions are equivalent:
\begin{itemize}
\item[(i)] $f$ has the constant ratio property with respect to ${\cal Z}$;
\item[(ii)]  $\|{\cal Z}^T_f\|$ (or equivalently, $\|{\cal Z}^\perp_f\|$) is constant along $M^n$;
\item[(iii)] $A^f_{{\cal Z}^\perp_f}{\cal Z}^T_f=0$;
\item[(iv)] Either ${\cal Z}^T_f$ is everywhere vanishing or the integral curves of ${\cal Z}^T_f$ are geodesics.
\end{itemize} 
In particular, if $n=m+1$ then the constant ratio property with respect to ${\cal Z}$ 
implies the principal direction property with respect  to ${\cal Z}$, if the constant value of $\|{\cal Z}^T_f\|$
is nonzero.
\end{lemma}
\begin{proof}
Taking the tangent and normal components of 
$\tilde \nabla_X {\cal Z}=0$
for  $X\in \mathfrak{X}(M)$, where $\tilde \nabla$ is the Levi-Civita connection of $N^n$, and using the Gauss and 
Weingarten equations of $f$, yield 
\begin{equation}\label{eq:NablaT}
\nabla_X {\cal Z}^T_f= A^f_{{\cal Z}^\perp_f} X
\end{equation}
and 
\begin{eqnarray}\label{eq:Derivadaeta}
\alpha_{ f}(X,{\cal Z}^T_f)=-\nabla_X^\perp {\cal Z}^\perp_f,
\end{eqnarray}
 where $\nabla$ is the Levi-Civita connection of $M^m$. The first  
assertion in the statement is an immediate  consequence of (\ref{eq:Derivadaeta}), and the second 
one is clear, for ${\cal Z}$ has constant length. The latter observation also gives the equivalence 
between $(i)$ and $(ii)$. By \eqref{eq:Derivadaeta}, conditions 
$(ii)$ and  $(iii)$ are equivalent, and the latter holds if and only if $\nabla_{{\cal Z}^T_f}{\cal Z}^T_f=0$ 
by \eqref{eq:NablaT}, which yields the equivalence between $(iii)$ and $(iv)$. The last assertion is now 
an immediate consequence of either $(ii)$ or $(iii)$.
\end{proof}

The next lemma plays a key role in the sequel.
  
\begin{lemma}\label{le:basic}
Let $f\colon M^m\to N^n$ be an isometric immersion, let $\Psi\colon N^n\to\hat N^n$
be a local conformal diffeomorphism and set $\hat f=\Psi\circ f$. Let 
${\cal Z}$ and $\hat {\cal Z}$ be vector fields on $N^n$ and $\hat{N}^n$ 
that are $\Psi$-related, that is, $\hat {\cal Z}(\Psi(y))=\Psi_*(y){\cal Z}(y)$ 
for all $y\in N^n$. Then the following assertions hold:
\begin{itemize}
\item[(i)] $\hat{\cal Z}^T_{\hat f}(x)={\cal Z}^T_f(x)$ and 
$\hat{\cal Z}^\perp_{\hat f}(x)=\Psi_*(f(x)){\cal Z}^\perp_f(x)$ for all 
$x\in M^m$;
\item[(ii)] if ${\cal Z}^T_f$ is nowhere vanishing, then the ratios $\|{\cal Z}^\perp_f\|/\|{\cal Z}^T_f\|$ 
and $\|\hat{\cal Z}^\perp_{\hat f}\|/\|\hat{\cal Z}^T_{\hat f}\|$ coincide;
\item[(iii)] $f$ has the constant ratio property with respect to  
${\cal Z}$ if and only if $\hat f$ has the constant ratio property with respect to  
$\hat {\cal Z}$; 
\item[(iv)] $f$ has the principal direction property with respect to ${\cal Z}$  
if and only if $\hat f$ has the principal direction property with respect to 
$\hat{\cal Z}$.
\end{itemize}
\end{lemma}
\begin{proof} 
The assertion in $i)$ follows from
\begin{eqnarray*}
\hat f_*(x)\hat{\cal Z}^T_{\hat f}(x)+\hat{\cal Z}^\perp_{\hat f}(x)&=&\hat {\cal Z}(\hat f(x))\\
  &=&\Psi_*(f(x)){\cal Z}(f(x))\\
&=&\Psi_*(f(x))f_*(x){\cal Z}^T_f(x)+\Psi_*(f(x)){\cal Z}^\perp_f(x)\\
&=&(\Psi\circ f)_*(x){\cal Z}^T_f(x)+\Psi_*(f(x)){\cal Z}^\perp_f(x)\\
&=&\hat f_*(x){\cal Z}^T_{ f}(x)+\Psi_*(f(x)){\cal Z}^\perp_f(x).
\end{eqnarray*}
Now, if $\va\in C^{\infty}(N)$ is the conformal factor of $\Psi$, that is,
$$\<\Psi_*X, \Psi_*Y\>=\va^2(y)\<X, Y\>$$ for all $y\in N^n$ and $X, Y\in T_yN$, then
\begin{eqnarray*}\<\hat{\cal Z}^T_{\hat f}(x),\hat{\cal Z}^T_{\hat f}(x)\>_{\hat f}
&=&\<\hat f_*(x)\hat{\cal Z}^T_{\hat f}(x), \hat f_*(x)\hat{\cal Z}^T_{\hat f}(x)\>\\
&=&\<\Psi_*(f(x)) f_*(x){\cal Z}^T_f(x), \Psi_*(f(x)) f_*(x){\cal Z}^T_f(x)\>\\
&=&(\va(f(x))^2\<f_*(x){\cal Z}^T_f(x),  f_*(x){\cal Z}^T_f(x)\>\\
&=&(\va(f(x))^2\<{\cal Z}^T_f(x), {\cal Z}^T_f(x)\>_f
\end{eqnarray*}
and 
\begin{eqnarray*}\<\hat{\cal Z}^\perp_{\hat f}(x),\hat{\cal Z}^\perp_{\hat f}(x)\>
&=&\<\Psi_*(f(x)){\cal Z}^\perp_f(x),\Psi_*(f(x)){\cal Z}^\perp_f(x)\>\\
&=&(\va(f(x))^2\<{\cal Z}^\perp_f(x), {\cal Z}^\perp_f(x)\>
\end{eqnarray*}
for all $x\in M^m$, where $\<\;, \;\>_f$ and $\<\;, \;\>_{\hat f}$ denote the metrics 
induced by $f$ and $\hat f$, respectively.
The assertion in $(ii)$ is an immediate consequence of the preceding relations,
and $(iii)$ follows by combining $(i)$ and $(ii)$. Finally, $(iv)$ follows from 
the  relation 
\[
\alpha_{\hat f}(X,Y)=
\Psi_*\alpha_f(X,Y)-\frac{1}{\varphi}\<X,Y\>\Psi_*(\grad \varphi)^\perp
\]
between the second fundamental forms $\alpha_f$, $\alpha_{\hat f}$ 
of $f$ and $\hat f$, respectively. 
\end{proof}

  As a consequence of Lemmas \ref{prop:paralellvf} and \ref{le:basic},
for a hypersurface $f\colon M^m \to N^{m+1}$ the constant ratio property with respect 
to ${\cal Z}\in \mathfrak{X}(N)$ implies the principal direction property with 
respect  to ${\cal Z}$ (if ${\cal Z}^T_f$ is nowhere vanishing) also when ${\cal Z}$ 
is $\Psi$-related to a parallel vector field $\hat {\cal Z}\in \mathfrak{X}(\hat N)$ 
under a (local) conformal diffeomorphism $\Psi\colon N^{m+1}\to \hat N^{m+1}$.
If $f\colon M^m \to N^{n}$ has arbitrary codimension, the two properties are related 
as follows.

\begin{corollary}\label{prop:charac2} 
Let $N^{n}$ be a Riemannian manifold and let ${\cal Z}\in \mathfrak{X}(N)$ 
be $\Psi$-related to a parallel vector field $\hat{\cal Z}\in\mathfrak{X}(\hat N)$
under a (local) conformal diffeomorphism $\Psi\colon N^{n}\to \hat N^{n}$.
Then an isometric immersion $f\colon M^m\to N^{n}$  has  the  principal direction
property with respect to ${\cal Z}$ if and only if ${\cal Z}^T_f$ is nowhere vanishing, 
the ratio $\|{\cal Z}^\perp_f\|/\|{\cal Z}^T_f\|$ is constant along 
$\{{\cal Z}^T_f\}^\perp$ and either ${\cal Z}^\perp_f$ is everywhere vanishing or 
$\zeta_f= {\cal Z}^\perp_f/\|{\cal Z}^\perp_f\|$ is 
parallel in the normal connection along $\{{\cal Z}^T_f\}^\perp$.
\end{corollary}
\begin{proof}
Let $\hat f=\Psi\circ f\colon M^m\to \hat N^{n}$. By the first assertion in
Lemma \ref{prop:paralellvf}, $\hat f$ has the  principal direction property with 
respect to $\hat{\cal Z}$ if and only if $\hat{\cal Z}^T_f$ is nowhere vanishing and
$\hat{\cal Z}^\perp_{\hat f}$ is parallel in the normal connection along 
$\{\hat{\cal Z}^T_{\hat f}\}^\perp$. 
The latter condition is equivalent to  
$\|\hat{\cal Z}^\perp_{\hat f}\|$ being  constant along 
$\{\hat{\cal Z}^T_{\hat f}\}^\perp$ and, if this constant is nonzero, to the unit vector 
field $\hat{\cal Z}^\perp_{\hat f}/\|\hat{\cal Z}^\perp_{\hat f}\|$ being parallel in 
the normal connection along $\{\hat{\cal Z}^T_{\hat f}\}^\perp$. 
Since $\hat{\cal Z}$ has constant length, $\|\hat{\cal Z}^\perp_{\hat f}\|$ 
being a nonzero constant along $\{\hat{\cal Z}^T_{\hat f}\}^\perp$ is  equivalent to 
the ratio $\|\hat{\cal Z}^\perp_{\hat f}\|/\|\hat{\cal Z}^T_{\hat f}\|$ being 
constant along $\{\hat{\cal Z}^T_{\hat f}\}^\perp$.
Now, since $\Psi\colon N^{n}\to \hat N^{n}$ is a (local) conformal 
diffeomorphism,  $\hat\nabla_X^\perp \hat{\cal Z}^\perp_{\hat f}$ and 
$\nabla_X^\perp {\cal Z}^\perp_{f}$ are related by
\begin{eqnarray}\label{eq:nderiv}
\hat\nabla_X^\perp \hat{\cal Z}^\perp_{\hat f}&=&\Psi_* \, 
\nabla_X^\perp {\cal Z}^\perp_{f} + \frac{1}{\varphi\circ f} 
X(\varphi\circ f)\Psi_*{\cal Z}^\perp_{f}\nonumber\vspace{1ex}\\
&=&\Psi_* \, 
\nabla_X^\perp {\cal Z}^\perp_{f} 
+ \frac{1}{\varphi\circ f} 
X(\varphi\circ f)\hat{\cal Z}^\perp_{\hat f}
\end{eqnarray}
where, as above, $\varphi$ is the conformal factor of $\Psi$. 
Therefore, if $\hat{\cal Z}^\perp_{\hat f}$ is nowhere vanishing, then the unit vector 
field $\hat{\cal Z}^\perp_{\hat f}/\|\hat{\cal Z}^\perp_{\hat f}\|$ is parallel in the 
normal connection along $\{\hat{\cal Z}^T_{\hat f}\}^\perp$ if and only 
if the unit vector field ${\cal Z}^\perp_{f}/\|{\cal Z}^\perp_{f}\|$ is parallel 
in the normal connection along $\{{\cal Z}^T_{f}\}^\perp$. Moreover, 
by parts $(i)$ and $(ii)$ of Lemma \ref{le:basic}, if $\hat{\cal Z}^T_{\hat f}$ is
nowhere vanishing, then the ratio 
$\|\hat{\cal Z}^\perp_{\hat f}\|/\|\hat{\cal Z}^T_{\hat f}\|$ is
constant along $\{\hat{\cal Z}^T_{\hat f}\}^\perp$ if and only if the ratio 
$\|{\cal Z}^\perp_{f}\|/\|{\cal Z}^T_{f}\|$ is constant along 
$\{{\cal Z}^T_{f}\}^\perp$. 
\end{proof}
 
\section{Constant ratio  property  with respect to $\frac{\partial}{\partial t}$}
\label{crpconst}

In this section we investigate the isometric immersions 
$f\colon\, M^m \to \Q_\epsilon^n\times \R$, with $m\geq 2$, into the Riemannian product 
of $\Q_\epsilon^n$ and $\R$, that have the 
constant ratio property with respect to the unit vector field 
${\cal Z}=\frac{\partial}{\partial t}$, tangent to the factor $\R$. We may assume that
$\epsilon \in \{-1,0,1\}$, and we will use the following models for $\Q_0^n$, $ \Q_1^n$ 
and $ \Q_{-1}^n$, respectively:
\begin{align}
& \R^n , \label{modelQ0} \\
& \Sf^n = \left\{ (x_1,\ldots,x_{n+1}) \in \R^{n+1}: x_1^2 + \ldots + x_{n+1}^2 = 1 \right\}, \label{modelQ1} \\
& \Hy^n = \left\{ (x_1,\ldots,x_{n+1}) \in \R_1^{n+1}:x_1^2 + \ldots + x_n^2 - x_{n+1}^2 = -1, x_{n+1} > 0 \right\}, 
\label{modelQ-1}
\end{align}
which immediately give models for $\Q^n_0 \times \R$, $\Q^n_1 \times \R$ and 
$\Q^n_{-1} \times \R$ as subsets of $\R^{n+1}$, $\R^{n+2}$ and $\R^{n+2}_1$, respectively. 
In all cases, the last coordinate on the ambient space coincides with the coordinate on the 
factor $\R$ in $\Q^n_{\epsilon} \times \R$, which we always denote by $t$. 
In particular, the vector field $\frac{\partial}{\partial t}$ can be seen as a constant unit 
vector field on $\R^{n+1}$, $\R^{n+2}$ or $\R^{n+2}_1$, respectively, and for $\epsilon=0$ 
the problem at hand comes down to finding isometric immersions into $\R^{n+1}$ having the 
constant ratio property with respect to a constant vector field.

Note that, in all three cases, two classes of trivial examples of submanifolds of $\Q_\epsilon^n\times \R$ 
with the constant ratio property with respect to ${\cal Z}=\frac{\partial}{\partial t}$ occur. 
First, if $N^{m-1}$ is a submanifold of $\Q_\e^n$, then any open subset of the product immersion 
$N^{m-1}\times \R$ into $\Q_\epsilon^n\times \R$ has everywhere vanishing normal vector field 
${\cal Z}_f^\perp$. These examples are called \emph{vertical cylinders}. 
On the other hand, any submanifold of a \emph{horizontal slice} $\Q_\epsilon^n\times \{t_0\}$, 
where $t_0\in \R$ is fixed, has everywhere vanishing tangent vector field ${\cal Z}_f^T$. 
In the following we describe how any other example arises. 

Let $M^m=J\times N^{m-1}$ be a product manifold endowed with a
polar metric
\be\label{eq:polarmet2}
d\sigma^2=\pi_1^*ds^2+\pi_2^*g_s,
\ee
where $\pi_1\colon M^m\to J$ and 
$\pi_2\colon M^m\to N^{m-1}$ are the canonical projections, $ds^2$ is 
the standard metric on $J$ and $\{g_s\}_{s\in J}$ is a one-parameter 
family of metrics on $N^{m-1}$ indexed on $J$.  That the metric of 
$M^m=J\times N^{m-1}$ is the polar metric (\ref{eq:polarmet2}) is equivalent
to requiring the metric to be orthogonal and the curves $s\mapsto (s,x)$, $x\in N^{m-1}$, 
to be geodesics in $M^m$ (see Proposition $2.3$ of \cite{to2}).
A particular case of a polar metric (\ref{eq:polarmet2}) is a warped product
metric $$\pi_1^*ds^2+(\rho\circ \pi_1)^2\pi_2^*g$$
for some $\rho\in C^{\infty}(J)$ and some fixed metric $g$ on $N^{m-1}$, which 
corresponds to the case in which all metrics $g_s$, $s\in J$, are
homothetical to a fixed metric $g$ on $N^{m-1}$.  

\begin{theorem}\label{thm:CRmain} 
Let $\phi\colon M^m\to \Q_\epsilon^n\subset \R^{n+1}_{(1)}$, $m\geq 2$, 
$\epsilon \in \{-1, 0, 1\}$, be an isometric immersion of a product manifold
$M^m=I\times N^{m-1}$ endowed with a polar metric. Then the map
$f\colon\, M^m\to \Q_\epsilon^n\times \R$ given by
\be\label{eq:fCRddt}
f(s,x)=(\phi(s,x), As),\;\;A\neq 0,
\ee
is an immersion with the constant 
ratio property with respect to the unit vector field 
${\cal Z}=\frac{\partial}{\partial t}$ tangent to the factor $\R$.

Conversely, if $f\colon\, M^m\to \Q_\epsilon^n\times \R$, $m\geq 2$, 
$\epsilon \in \{-1, 0, 1\}$, is an  isometric immersion with  the constant 
ratio property with respect to ${\cal Z}=\frac{\partial}{\partial t}$, then 
either $f(M^m)$ is an open part of a vertical cylinder, or it is contained 
in a horizontal slice, or $f$ is locally given as above 
(globally, with $J=\R$, if the integral curves of ${\cal Z}_{ f}^T$ 
are defined on $\R$).
\end{theorem} 

 If $m=n$, then Theorem \ref{thm:CRmain} reduces to Corollary $2$ of \cite{to},
 by taking into account the following elementary observation
 that we state as a lemma.
 
 \begin{lemma}\label{le:element} If $\phi\colon M^n\to \Q_\epsilon^n\subset \R^{n+1}_{(1)}$, 
 with $n\geq 2$ and $\epsilon \in \{-1, 0, 1\}$, is a local isometry of a product manifold
$M^n=I\times N^{n-1}$ endowed with a polar metric, then
 there exists a  hypersurface $\phi_0:N^{n-1}\to\Q_\epsilon^n$ 
such that  
${\phi}_s:= \phi(s, \cdot) \colon N^{n-1}\to\Q_\epsilon^n$ 
is the family of parallel hypersurfaces 
to $\phi_0$ indexed on  $J$, that is, 
\[
\phi_s(x)=C_\e(s) \phi_0(x)+S_\e(s)N(x),
\]
where $N$ is a unit normal vector field to $\phi_0$ and $(C_\e(s),S_\e(s))$ 
is either $(\cos s,\sin s)$, $(1, s)$ or $(\cosh s, \sinh s)$, depending on 
whether $\epsilon=1$, $0$ or $-1$, respectively. 
\end{lemma}

Before giving the proof of Theorem \ref{thm:CRmain} we show the following fact,
which sheds light on how product manifolds $M^m=J\times N^{m-1}$ with a polar 
metric arise.

\begin{lemma}\label{le:constgrad}
Let $M^m$ be a Riemannian manifold and let $\rho\in C^{\infty}(M)$ be such that
$\textup{\grad} \rho$ has unit length. Then $M^m$ is locally isometric to 
a product manifold $J\times N^{m-1}$ endowed with a polar metric (\ref{eq:polarmet2}),
with the curves $s\mapsto (s, x_0)$, $x_0\in N^{m-1}$, corresponding to the 
integral curves of $\textup{\grad} \rho$. The assertion holds globally, with $J=\R$,
if the integral curves of $\textup{\grad} \rho$ are defined on $\R$.
\end{lemma}
\proof Denote ${\cal T}=\textup{\grad} \rho$. Then
$\<\nabla_{\cal T}{\cal T},X\>=\<\nabla_X{\cal T}, {\cal T}\>=(1/2)X\<{\cal T},{\cal T}\>=0$
for all $X\in \mathfrak{X}(M)$, thus the integral curves of ${\cal T}$ are geodesics 
of $M^m$. On the other hand, since ${\cal T}$ is a gradient vetor field, then the distribution
$\{{\cal T}\}^\perp$ is integrable, its leaves being the level sets of $\rho$.
By Theorem~$2.7$ of \cite{to2}, $M^m$ is locally isometric to a product 
$J\times N^{m-1}$ endowed with a polar metric (\ref{eq:polarmet2}). 
Moreover, also by that theorem, $M^m$ is globally isometric to a product 
$\R\times N^{m-1}$ endowed with such a polar metric if the integral curves of 
${\cal T}$ are complete geodesics.\vspace{2ex}\qed

We point out that \emph{any} Riemannian manifold $M^m$ admits \emph{locally} a function
$\rho$ whose gradient has unit length. For instance, if $x\in M^m$  and $U\subset T_xM$
is  an open neighborhood of $0\in T_xM$ restricted to which the exponential map $\exp_x$
is a diffeomorphism onto $V\subset M^m$, then the map $\rho\colon V\setminus \{x\}\to \R$
given by $\rho(y)=\|\exp^{-1}_x y\|$ has this property. The integral curves of $\grad \rho$
are radial geodesics issuing from $x$ and the leaves of the orthogonal distribution
are the geodesic hyperspheres centered at $x$.\vspace{2ex}

Therefore, starting with \emph{any} isometric immersion $f\colon M^m\to \Q_\epsilon^n$, $m\geq 2$, 
$\epsilon \in \{-1, 0, 1\}$, one can produce
an immersion  $\phi\colon J\times N^{m-1}\to \Q_\epsilon^n$  
of a product manifold $J\times N^{m-1}$, 
with $J\subset \R$ an open interval, whose induced metric is the polar 
metric (\ref{eq:polarmet2}). It suffices to compose a map $\psi\colon J\times N^{m-1}\to U$ 
onto an open subset $U$ of $M^m$ whose induced metric is such a polar metric, as the isometry 
in the preceding paragraph, with the restriction $f|_U$ of $f$ to $U$.  \vspace{2ex}\\
\noindent \emph{Proof of Theorem \ref{thm:CRmain}:} 
Since ${\cal Z}=\frac{\partial}{\partial t}$ has  constant length,  
proving that  $f$ has the constant ratio property with respect to ${\cal Z}$ is equivalent 
to showing that either ${\cal Z}^T_{f}$ or ${\cal Z}^\perp_{f}$ has constant length. 
From \eqref{eq:fCRddt} we have
\[
0=X \<f,{\cal Z}\> = \<f_*X, {\cal Z}\> = \<X, {\cal Z}^T_{f}\>_f
\]
for all $X\in \mathfrak{X}(N)$, where $\<\,\cdot\,, \,\cdot \,\>_f$ denotes the metric 
$d\sigma^2=(1+A^2)\pi_1^*ds^2+\pi_2^*g_s$ induced by $f$.
It follows that ${\cal Z}^T_{f}=\lambda\frac{\partial}{\partial s}$ 
for some $\lambda\in C^{\infty}(M)$. On the other hand,
$$
A=\frac{\partial}{\partial s}\<f,{\cal Z}\> = \<f_*\frac{\partial}{\partial s}, {\cal Z}\> 
 = \< \frac{\partial}{\partial s}, {\cal Z}^T_{f} \>_f 
 = \lambda \< \frac{\partial}{\partial s}, \frac{\partial}{\partial s} \>_f = (1+A^2)\lambda.
$$
Therefore ${\cal Z}^T_{f}$ has constant length $|A|(1+A^2)^{-1/2}$.

We now prove the converse. If either ${\cal Z}^\perp_{f}$ or ${\cal Z}^T_{f}$ vanishes identically, 
then $f(M)$ is an open part of a vertical cylinder or it is contained in a horizontal slice,
respectively. 
Suppose from now on that neither of these possibilities occur. Denoting $F=\<f,{\cal Z}\>$ we have 
 $$X(F) = X\<f, {\cal Z}\> = \< f_*X,{\cal Z} \> = \<X,{\cal Z}^T_{f}\>_f$$
for all $X \in \mathfrak{X}(M)$. This means that ${\cal Z}^T_{f}$ is the gradient of $F$. 
Since $f$ has the constant ratio property with respect to ${\cal Z}=\frac{\partial}{\partial t}$,
which has constant length, also ${\cal Z}^T_{f}$ has constant length. It follows from Lemma \ref{le:constgrad}
that $M^m$ is locally isometric to a product $M^m=I\times N^{m-1}$ endowed with a polar 
metric (\ref{eq:polarmet2}), and actually globally  if the integral curves of 
${\cal Z}^T_{f}$ are complete geodesics.
Furthermore, from
$$
\hess F({\cal Z}^T_{f},{\cal Z}^T_{f}) = \< \nabla_{{\cal Z}^T_{f}}{\cal Z}^T_{f}, {\cal Z}^T_{f} \>_f = 0,
$$
it follows that $F$ depends linearly on $s$. \qed
   
\subsection{The warped product case}

We now consider isometric immersions $f\colon M^m\to I\times_{\rho}\Q^n_\e$ into a 
warped product $I\times_{\rho}\Q^n_\e$, where $I\subset \R$ is an open interval,
 endowed with the warped product metric
\[
\<\; ,\;\> = \pi_1^\ast dt^2+(\rho\circ \pi_1)^2\pi_2^\ast\<\; ,\;\>_{\Q^n_\e},
\]
that have the constant ratio property with respect to the unit vector field 
$\mathcal Z = \frac{\partial}{\partial t}$ tangent to the factor $I$. 
Here the warping function $\rho\colon I\to\R$ is any smooth positive function and 
$\pi_1\colon I\times\Q^n_\e\to I$, $\pi_2\colon I\times\Q^n_\e\to \Q^n_\e$ denote
the canonical projections. 

The classification of such isometric immersions  will follow from part $(iii)$ of
Lemma~\ref{le:basic} and Theorem \ref{thm:CRmain} by remarking that $I\times_{\rho}\Q^n_\e$
is conformally diffeomorphic to an open subset of $\Q^n_\e\times\R$. This
 observation  goes back to Gerardus Mercator (1512--1594) (before calculus was invented!)
 in his construction of a conformal map of
the sphere onto the plane from the standard parametrization of the sphere by the latitude and longitude.
Namely, let $G\colon I\to \R$ be any primitive of the function 
$\displaystyle{\frac{1}{\rho}\colon I\to \R}$, that is, 
$\displaystyle{
G'(t)= \frac{1}{\rho(t)}}
$
for all $t\in I$. Then $G$ is an increasing diffeomorphism onto an open interval $J\subset \R$, 
and the map $\phi:I\times_\rho\Q_\e^n\to\Q_\e^n\times J\subset \Q_\e^n\times \R$,
defined as
\[
\phi(s,x) = (x,G(s)),
\]
is a conformal diffeomorphism between $I\times_\rho\Q_\e^n$ and
$\Q_\e^n\times J$ with conformal factor $\lambda=\rho\circ \pi_1$, that is, 
its induced metric is $\lambda^2\<\; ,\;\>$.
Its inverse
\be\label{eq:psi}
\psi = \phi^{-1}:\Q_\e^n\times J\to I\times_\rho\Q_\e^n
\ee
is given by 
\be\label{eq:psib}
\psi(x,t)=(F(t),x)
\ee
where $F=G^{-1}$, so that $F'(t)=\rho(F(t))$ for all $t\in J$. 
Moreover, 
\[
\Psi_\ast(x_0,t_0) \frac{\partial}{\partial t}(x_0,t_0) = 
F'(t_0) \frac{\partial}{\partial  t}(F(t_0),x_0) = 
\rho(F(t_0))\frac{\partial}{\partial  t}(F(t_0),x_0)
\]
for all $(x_0,t_0)\in\Q^n_{\epsilon}\times J$, where the former 
$\frac{\partial}{\partial t}$ denotes a unit tangent vector field
on $J$, while the latter two denote a unit tangent vector field
on $I$. Therefore, the following result is a consequence of part 
$(iii)$ of Lemma \ref{le:basic} together with Theorem \ref{thm:CRmain}.

\begin{theorem}\label{thm:CRwarped} 
Let $I \times_{\rho}\Q^n_\e$ be a warped product of an open interval 
$I\subset\R$ and $\Q^n_\e$, $\e\in\{-1, 0, 1\}$.
Let $F\colon \tilde{J}\to \tilde{I}$ be a diffeomorphism of an open interval 
$\tilde{J}\subset \R$ onto an open interval $\tilde{I}\subset I$ satisfying 
\be\label{eq:ode}
F'(t)=\rho(F(t))
\ee
for all $t\in \tilde{J}$, let $A\neq 0$ and let 
 $J= h^{-1}(\tilde J)$, where $h\colon \R \to \R$ is the linear map
given by $h(t)= A t$ for all $t\in \R$. 
Let $\phi\colon M^m\to \Q_\epsilon^n$ be an isometric
immersion of the product manifold $M^m={J}\times N^{m-1}$, $m\geq 2$, 
endowed with a polar metric (\ref{eq:polarmet2}). Then the map 
$f\colon\, M^m\to I \times_{\rho} \Q_\epsilon^n$ given by
\begin{equation} \label{fwarped}
f(s,x)=(F(As), \phi(s,x))
\end{equation}
is an immersion with 
the constant ratio property with respect to the unit vector field 
${\cal Z}=\frac{\partial}{\partial t}$ tangent to the factor $I$.

Conversely, if $f\colon\, M^m \to I \times_{\rho} \Q^n_{\epsilon}$, $m\geq 2$, 
$\epsilon \in \{-1, 0, 1\}$, is an  isometric immersion that has  the constant ratio 
property with respect to ${\cal Z}=\frac{\partial}{\partial t}$, then either $f(M)$ is an 
open part of a vertical cylinder $I \times_{\rho} N^{m-1}$ or of a submanifold of a 
horizontal slice, or $f$ is locally given as above (globally, if the integral curves
 of ${\cal Z}_{ f}^T$ are defined on $\R$).
\end{theorem} 

In the case of hypersurfaces one has a more explicit description by taking Lemma \ref{le:element}
into account.

\subsection{The space form case -- loxodromic immersions}
\label{loxodromic}

There are well-known representations of any real space form 
$\Q^{n+1}_{\epsilon}$, $\epsilon\in\{1,0-1\}$ (or an open dense subset of it),
 as a warped product of an open interval 
and another space form, namely,
\begin{align}
&\R^{n+1}\setminus \{0\}= (0,+\infty) \times_{\rho} \Sf^n, && \mbox{with} \ \rho(t)=t, \label{warpedmodQ0} \\
& \Sf^{n+1}\setminus \{S,N\}=(0,\pi) \times_{\rho} \Sf^n, && \mbox{with} \ \rho(t)=\sin t, \label{warpedmodQ1} \\
&\Hy^{n+1}\setminus\{P\}= (0,+\infty) \times_{\rho} \Sf^n, && \mbox{with} \ \rho(t)=\sinh t, \label{warpedmodQ-1a}\\
&\Hy^{n+1}= \R \times_{\rho} \Hy^n, && \mbox{with} \ \rho(t)=\cosh t,\label{warpedmodQ-1b}\\
&\Hy^{n+1}= \R \times_{\rho} \R^n, && \mbox{with} \ \rho(t)=e^t/\sqrt{2}. \label{warpedmodQ-1c}
\end{align}
The first one is just the standard  description of $\R^{n+1}\setminus \{0\}$ in spherical coordinates, given by the isometry
\begin{equation} \label{mapWPeuclid}
(0,\infty) \times_{id} \Sf^n \to \R^{n+1}\setminus \{0\}: (t,x) \mapsto tx.
\end{equation}
The second one corresponds to the isometry of the unit sphere $\Sf^{n+1}$ minus the south and north poles
$S=(0,\ldots,0,-1)$ and $N=(0,\ldots,0,1)$ given, with respect to the model \eqref{modelQ1} for $\Sf^n$ 
and the corresponding one for $\Sf^{n+1}$, by the map
\begin{equation} \label{mapWPsphere}
(0,\pi) \times_{\sin} \Sf^n \to \Sf^{n+1}\setminus \{S,N\}: (t,x) \mapsto 
(\sin t)x + (\cos t)N.
\end{equation}
In terms of the model \eqref{modelQ1} for $\Sf^n$ and the one corresponding to \eqref{modelQ-1} for 
$\Hy^{n+1}$, the representation (\ref{warpedmodQ-1a}) of $\Hy^{n+1}$ minus $P=(0,\ldots,0,1)$ is given by the  
isometry 
\begin{equation} \label{mapWPhypa}
(0,+\infty) \times_{\sinh} \Sf^n \to \Hy^{n+1}\setminus \{P\} : (t,x) \mapsto (\sinh t)x + (\cosh t)P.
\end{equation}
For the representation (\ref{warpedmodQ-1b}) of $\Hy^{n+1}$ take  the  
isometry 
\begin{equation} \label{mapWPhypb}
\R \times_{\cosh} \Hy^n \to \Hy^{n+1} : (t,x) \mapsto (\sinh t)e + (\cosh t)x
\end{equation}
with $\Hy^n\subset \R_1^{n+1}=\{e\}^\perp \subset \R_1^{n+2}$, whereas for (\ref{warpedmodQ-1c}) the  
isometry is given by
\begin{equation} \label{mapWPhypc}
\R \times_{(1/\sqrt{2})\exp} \R^n \to \Hy^{n+1} : (t,x) \mapsto (1/\sqrt{2})e^t(v+x-(1/2)\|x\|^2w) - (1/\sqrt{2})e^{-t}w
\end{equation}
where $\{v, w\}$ is a pseudo-orthonormal basis of $(\R^n)^\perp\subset \R_1^{n+2}$, that is, 
$\<v,v\>=0=\<w,w\>$ and  $\<v, w\>=1$.\vspace{1ex}

Classicaly, a unit speed curve $\gamma\colon I\to \Sf^2$ is called a \emph{loxodrome} 
if its tangent vector makes a constant angle with the meridians of $\Sf^2$. 
One can similarly define a loxodrome $\gamma\colon I\to \Sf^{n+1}$
in $\Sf^{n+1}$ for any $n\geq 2$. 
Using the model \eqref{warpedmodQ1} to describe $\Sf^{n+1}$, 
loxodromes in $\Sf^{n+1}$ are precisely the images under the isometry in (\ref{mapWPsphere}) 
of the curves into $(0,\pi) \times_{\sin} \Sf^n$ that have the constant ratio property with 
respect to the unit tangent vector $\mathcal Z = \frac{\partial}{\partial t}$ to $(0,\pi)$. 

  Let us denote by $\phi$ either of the isometries given by $(\ref{mapWPeuclid})$ to 
$(\ref{mapWPhypc})$, and by $\frac{\partial}{\partial t}$ a unit vector field tangent to the 
one-dimensional factor of the corresponding warped-product space. Then the vector field 
$\phi_*\frac{\partial}{\partial t}$ in $\Q_\epsilon^{n+1}$ gives rise to a ``radial" vector field 
${\cal R}$ in $\Q_\epsilon^{n+1}$ (of three different types for $\epsilon=-1$). 
We will also refer to an isometric immersion into any of the models 
\eqref{warpedmodQ0}--\eqref{warpedmodQ-1c} having the constant ratio 
property with respect to the radial vector field ${\cal R}$ as a 
\emph{loxodromic isometric immersion}. To distinguish the three possible cases in $\Hy^{n+1}$, 
we will call a loxodromic isometric immersion $f\colon M^m\to \Hy^{n+1}$ \emph{elliptic}, 
\emph{hyperbolic} or \emph{parabolic}, depending on whether the radial vector field ${\cal R}$ 
is determined by an isometry $\phi$ as in \eqref{warpedmodQ-1a}, \eqref{warpedmodQ-1b} or 
\eqref{warpedmodQ-1c}, respectively.

For $\epsilon=0$, the vector field ${\cal R}$ is the standard radial vector field,  and a 
better approach to studying the constant ratio property with respect to that vector field will 
be given in Section \ref{sec:CRradial}. Therefore, we will omit this case here.

A description of all loxodromic isometric immersions $f\colon M^m\to \Q_\epsilon^{n+1}$ 
for $\epsilon \in\{-1, 1\}$ follows immediately from Theorem \ref{thm:CRwarped}.
Namely, any loxodromic isometric immersion $f\colon M^m\to \Q_\epsilon^{n+1}$
is given by $f=\phi\circ \hat f$, where $\phi$ is one of the isometries given by 
(\ref{mapWPsphere}) to (\ref{mapWPhypc}) and $\hat f\colon M^n\to I\times_{\rho}\Q_\epsilon^{n}$ 
is given by (\ref{fwarped}), with $I=(0,\pi)$, $(0, \infty)$ or $\R$, and $F\colon \tilde J \to \tilde I$ a 
diffeomorphism of an open interval $\tilde J\subset \R$ onto an open interval $\tilde I\subset I$ 
satisfying (\ref{eq:ode}) for the restriction to $\tilde I$ of the corresponding warping function 
$\rho\in C^{\infty}(I)$ given as in either of equations $(\ref{warpedmodQ1})$ to $(\ref{warpedmodQ-1c})$.

For instance,  any solution of the ODE $F'(t)=\sin(F(t))$ on a maximal interval is given 
by $F:\R\to\R:t\mapsto 2\arctan(e^{t-c})$ for an arbitrary real constant $c$, which we can assume 
to be zero after a translation of the parameter $t$.
Therefore, any loxodromic isometric immersion into $\Sf^{n+1}$ is given by a map 
$f\colon M^m:=J\times N^{m-1}\to \Sf^{n+1}$ defined by
$$
f(s,x)=\sin\left(2\arctan(e^{(\sin\theta)s})\right)\phi(s,x)
+\cos\left(2\arctan(e^{(\sin\theta)s})\right)N$$
where $\phi\colon M^m\to \Sf^n$, $m\geq 2$, is an isometric  
immersion of the product manifold $M^m={J}\times N^{m-1}$ endowed with the polar metric 
(\ref{eq:polarmet2}).  

  Note that, although Theorem \ref{thm:CRwarped} is formulated for $m \geq 2$, 
we can still interpret it for $m=1$, i.e., to describe constant ratio curves. 
In this case, the manifold $N^{m-1}$ collapses to a point and $M^1=J$ is an open interval, 
so that the preceding equation reduces to
  $$\begin{array}{l}
f(s,x)=\sin\left(2\arctan(e^{(\sin\theta)s})\right)(\cos((\cos\theta)s), \sin((\cos\theta)s), 0)
\vspace{1ex}\\
\hspace*{10ex}+\cos\left(2\arctan(e^{(\sin\theta)s})\right)(0,0,1),\end{array} 
$$
which, under the reparametrizaton $u=2\arctan(e^{(\sin\theta)s})$, corresponds to the usual 
description of a loxodrome on $\Sf^2\subset \R^3$, namely,
$$ \alpha(u) = \left( \cos\left(\cot\theta\ln\left(\tan\frac{u}{2}\right)\right) \sin u, 
\sin\left(\cot\theta\ln\left(\tan\frac{u}{2}\right)\right) \sin u, \cos u \right). $$

   To obtain the parabolic loxodromic isometric immersions into $\Hy^{n+1}$, 
first notice that any solution of the ODE $F'(t)=\sqrt{2}e^{F(t)}$ on a maximal interval 
is given by 
$$
F\colon (-\infty, c/\sqrt{2}) \to\R:t\mapsto \log \frac{1}{c-\sqrt{2} t}
$$
for an arbitrary real constant $c$.
Therefore, any parabolic loxodromic isometric immersion into $\Hy^{n+1}$ 
is given by a map $f\colon M^m:=J\times N^{m-1}\to \Hy^{n+1}\subset  \R_1^{n+2}$ defined by
$$
\sqrt{2}f(s,x)=\frac{1}{c-\sqrt{2}(\sin \theta)s}\left(v+\phi(s,x)
+\frac{1}{2}\|\phi(s,x)\|^2w\right)-(c-\sqrt{2}(\sin \theta)s)w
$$
where $\{v, w\}$ is a pseudo-orthonormal basis of $(\R^n)^\perp\subset \R_1^{n+2}$, with
$\<v,v\>=0=\<w,w\>$ and  $\<v, w\>=1$, and $\phi\colon M^m\to \R^n$, $m\geq 2$, is an
isometric immersion of the product manifold $M^m={J}\times N^{m-1}$ endowed with the  
polar metric (\ref{eq:polarmet2}).  

Elliptic and hyperbolic loxodromic isometric immersions into $\Hy^{n+1}$ can be explicitly 
computed in a similar way.
     
\section{Principal direction property  with respect to $\frac{\partial}{\partial t}$} 
\label{pdpconst}

Isometric immersions $f:M^m\to\Q_\e^n\times\R$, with $m\geq 2$
and $\epsilon \in \{-1, 0, 1\}$, that satisfy the principal direction property
with respect to ${\cal Z}=\frac{\partial}{\partial t}$ and which are not 
vertical cylinders, were completely described in \cite{to} for the case
of hypersurfaces and in \cite{bt} for arbitrary codimension. The class of 
such isometric immersions was called  \emph{class ${\cal A}$} in \cite{bt}.
We will consider the case $\epsilon\in \{-1, 1\}$, the case $\epsilon=0$ 
being similar.
  
Let $\phi \colon\, N^{m-1}\to \Q_\e^n$  be an isometric immersion and assume 
that there exists an orthonormal set of normal vector fields 
$\xi_1, \ldots, \xi_k$ along $\phi$ that are parallel in the normal bundle. 
This assumption is satisfied, for instance, if $\phi$ has flat normal bundle. 
After including $\Q_\e^n$ into $\Q_\e^n \times \R$ in the standard way and 
using the models \eqref{modelQ1} and \eqref{modelQ-1}, we can look at $\phi$ 
as an immersion into $\R^{n+2}_{(1)}$ and at the vector fields 
$\xi_1, \ldots, \xi_k$ as vector fields on $\R^{n+2}_{(1)}$ along this immersion. 
Now add the following two vector fields along $\phi$ to this list:
$$ \xi_{k+1}(x) = \phi(x), \qquad \xi_{k+2}(x)
 = \frac{\partial}{\partial t}(\phi(x)) =(0,\ldots,0,1) $$
for all $x \in N^{m-1}$. Observe that $\xi_1,\ldots,\xi_{k+1}$ all have last 
component zero. The vector subbundle $E$ of  the normal bundle of 
$\phi:N^{m-1} \to \mathbb R^{n+2}_{(1)}$, 
whose fiber $E(x)$ at $x\in N^{m-1}$ is spanned by $\xi_1(x), \ldots, \xi_{k+2}(x)$, 
is parallel and flat and we may define a vector bundle isometry
\[
\Xi: N^{m-1}\times \R^{k+2}_{(1)} \to E
\]
by
\[
\Xi(x,(y_1,\ldots,y_{k+2})) = \sum_{i=1}^{k+2} y_i \xi_i(x).
\]
Now, let $\gamma:I\to\Q^k_\e\times\R\subset\R^{k+2}_{(1)}$, 
$\gamma= (\gamma_1,\ldots,\gamma_{k+1},\gamma_{k+2})$, 
be a smooth regular curve such that 
$\gamma_1^2+ \ldots + \gamma_k^2+\epsilon\gamma_{k+1}^2=\epsilon$
and $\gamma_{k+2}$ has non-vanishing derivative, and define a map
$f\colon N^{m-1} \times I \to \Q_\e^n\times \R$ by
\begin{eqnarray}\label{eq:fPD}
\begin{aligned}
f(x,s) & =  \Xi(x,\gamma(s)) = \sum_{i=1}^{k+2} \gamma_i(s)\xi_i(x) \\
& =  \sum_{i=1}^k \gamma_i(s)\xi_i(x) + \gamma_{k+1}(s)\phi(x) + 
(0,\ldots,0,\gamma_{k+2}(s)).
\end{aligned}
\end{eqnarray}
Notice that, if $\gamma$ parametrizes a line $\{p_0\} \times \R$ in 
$\Q_\e^k\times \R\subset \Ee^{k+2}$, then the map $f$ parametrizes a 
vertical cylinder. Geometrically, the image of $f$ is generated by parallel 
transporting the curve $\gamma$ in a product submanifold 
$\Q_\epsilon^k\times \R$ of a fixed normal  space of $\phi$ in 
$\mathbb R^{n+2}_{(1)}$ with respect to its normal connection. 

The following result from \cite{bt}, 
which we reformulate a bit for our purposes, states that the above 
construction gives all submanifolds of class $\mathcal A$, that is,
all submanifolds in $\Q_\e^n\times\R$ that have the principal direction property
with respect to ${\cal Z}=\frac{\partial}{\partial t}$.

\begin{theorem}[\cite{bt}]\label{thm:PDmain} 
The restriction  of the map 
$f \colon N^{m-1} \times I \to \Q_\e^n\times \R$ given by \eqref{eq:fPD}
to the subset of its regular points is an immersion with the principal 
direction property with respect to ${\cal Z}=\frac{\partial}{\partial t}$. 

Conversely, if $f\colon M^m\to \Q_\e^n\times \R$, $m\geq 2$, is an isometric immersion
with the principal direction property with respect to ${\cal Z}=\frac{\partial}{\partial t}$,
then it is locally given in this way (globally, if the  integral curves of 
${\cal Z}_{f}^T$ are defined on $\R$). 
\end{theorem}

The next result classifies isometric immersions $f:M^m\to\Q_\e^n\times\R$
with both the principal direction and constant ratio properties with respect 
to ${\cal Z}=\frac{\partial}{\partial t}$.

\begin{corollary}\label{cor:CRandPD} 
Let $f\colon M^m=N^{m-1}\times I\to \Q_\epsilon^n\times \R$ be given by \eqref{eq:fPD} 
for a regular curve 
$\gamma:I\subset\R\to\Q_\e^k\times\R\subset \R_{(1)}^{k+2}$,
$\gamma(s) = (\gamma_1(s),\ldots, \gamma_{k+1}(s), \gamma_{k+2}(s))$, such that 
$\bar \gamma\colon I\to \Q_\e^k\subset \R_{(1)}^{k+1}$, 
$\bar \gamma(s)= (\gamma_1(s),\ldots, \gamma_{k+1}(s))$, is a unit-speed curve and 
$\gamma_{k+2}(s))=As$ for some $A\neq 0$.  
Then the restriction of $f$ to the subset of its regular points is an  
immersion with both the principal direction and constant ratio properties 
with respect to ${\cal Z}=\frac{\partial}{\partial t}$.

Conversely,  any isometric immersion $f\colon M^m\to \Q_\e^n\times\R$, 
$m\geq 2$, with both the principal direction and constant ratio properties
with respect to ${\cal Z}=\frac{\partial}{\partial t}$,
is locally given in this way (globally, if the  integral curves of 
${\cal Z}_{ f}^T$ are defined on $\R$).
\end{corollary}

\proof Let $f\colon M^m=N^{m-1}\times I\to \Q_\epsilon^n\times \R\subset \R_{(1)}^{n+2}$
be the map given by \eqref{eq:fPD} for a regular curve 
$\gamma\colon I\subset\R\to\Q_\e^k\times\R\subset \R_{(1)}^{k+2}$,
$\gamma(s) = (\gamma_1(s),\ldots, \gamma_{k+1}(s), \gamma_{k+2}(s))$.
By \cite[Proposition 3.2]{bt}, the normal space $N_fM(x,s)$ 
of $f$ in $\R_{(1)}^{n+2}$ at $(x,s)$ is given by
\[
N_{f}{M}(x,s) = \Xi(x,\mathrm{span}\{\gamma'(s)\}^\perp) 
\subset N_{\phi}N(x),
\]
and the shape operator of $f$ with respect to $\Xi(x,\zeta)$ satisfies
\be\label{eq:alphapt3}
A^{f}_{\Xi(x,\zeta)}(x,s)\frac{\d}{\d s}
=\frac{\<\gamma''(s),\zeta\>}{\<\gamma'(s), \gamma'(s)\>}\frac{\d}{\d s}
\ee
for all $\zeta\in \R^{k+2}_{(1)}$ with $\<\zeta,\gamma'(s)\>=0$. 
On the other hand, we have
$$ {\cal Z}^\perp_{f} = \Xi(x,e_{k+2}^\perp),$$
where $e_{k+2}^\perp=e_{k+2}^\perp
-\<\gamma'(s), \gamma'(s)\>^{-1}\<e_{k+2}^\perp, \gamma'(s)\>\gamma'(s)$ 
is the normal component of $e_{k+2}$ along $\gamma$. 
It follows that $\<A_{{\cal Z}^\perp_{f}}{\cal Z}^T_{f}, {\cal Z}^T_{f}\>=0$ 
if and only if $\<\gamma''(s), e_{k+2}\>=0$ for all $s\in I$, 
that is, if and only if $\gamma_{k+2}$ is linear. \vspace{2ex}\qed

By Lemma \ref{prop:paralellvf}, special cases of isometric 
immersions $f: M^m\to \Q_\epsilon^n\times \R$ with both the principal 
direction and constant ratio properties
with respect to ${\cal Z}=\frac{\partial}{\partial t}$ are those for which 
${\cal Z}_f^\perp$ is parallel in the normal connection. These  isometric 
immersions were shown in \cite{bt} to be precisely the ones given by
\eqref{eq:fPD} with $\gamma\colon I\ \to \Q_\epsilon^k\times \R$  a 
geodesic of $\Q_\epsilon^k\times \R$.\vspace{1ex}

In the case of hypersurfaces, Theorem  \ref{thm:PDmain} reduces to 
Theorem $1$ in  \cite{to} by taking Lemma \ref{le:element}
into account. 

\subsection{The warped product case}

  The invariance of the principal direction property with respect to
conformal changes of the ambient metric (see part $(iv)$ of Lemma \ref{le:basic}) 
can be used to determine the isometric immersions 
$f\colon M^m\to I\times_{\rho}\Q^n_\e$ into a warped product $I\times_{\rho}\Q^n_\e$
that have the principal direction property with respect to the unit vector field 
$\mathcal Z = \frac{\partial}{\partial t}$ tangent to the factor $I$.

   Namely, let $\psi \colon \Q_\e^n\times J\to I\times_\rho\Q_\e^n$ be the conformal
diffeomorphism given by (\ref{eq:psib}). Then, by part $(iv)$ of 
Lemma \ref{le:basic} and Theorem \ref{thm:PDmain},  any such isometric immersion is 
given by $f=\psi\circ \hat f$, where $\hat f\colon M^m \to \Q^n_\e\times \R$ is the 
restriction of the map $\hat f\colon N^{m-1} \times I \to \Q_\e^n\times \R$ given by 
(\ref{eq:fPD}) to the subset $M^m\subset N^{m-1} \times I$ of its regular points, 
endowed with the induced metric, with the smooth regular curve 
$\gamma:I\to\Q^k_\e\times\R\subset\R^{k+2}_{(1)}$, 
$\gamma= (\gamma_1,\ldots,\gamma_{k+1},\gamma_{k+2})$, satisfying  
$\gamma(I)\subset \Q_\e^n\times J$, that is, $\gamma_{k+2}(I)\subset J$. Moreover,
choosing $\hat f$ as in Corollary \ref{cor:CRandPD}, the  isometric
immersions $f=\psi\circ \hat f$ give all isometric immersions into $I\times_\rho\Q_\e^n$
that have both the constant ratio and principal direction properties with respect to 
$\mathcal Z = \frac{\partial}{\partial t}$.

 In a similar way one can describe all isometric immersions $\tilde f\colon M^m\to \Q_\e^n$
that have the principal direction property (or both the constant ratio and principal direction 
properties) with respect to a radial vector field ${\cal R}$
as in Section \ref{loxodromic}. It suffices to compose the isometric immersions produced
as described in the previous paragraph with one of the isometries $\phi$ given by 
(\ref{mapWPsphere}) to (\ref{mapWPhypc}).

\section{Constant ratio property for radial vector fields} \label{sec:CRradial}

In this section and the next, we study Euclidean submanifolds with the constant 
ratio  or the principal direction property with respect to a radial 
vector field on $\R^n\setminus\{0\}$. In view of the remark after 
Definition \ref{def:cr}, without loss of generality we may work with the position 
vector field given by
$ \mathcal R(y) = y $
for all $y \in \R^n\setminus\{0\}$.

Our results will follow by combining the assertions in parts $(iii)$ and $(iv)$ of 
Lemma~\ref{le:basic} with the results of Sections \ref{crpconst} and \ref{pdpconst},
making use of the fact that the map 
\begin{equation} \label{eq:psiradial}
\Psi: \Sf^{n-1}\times\R \to \R^n \setminus\{0\}\; :\; (x,t) \mapsto e^t x 
\end{equation}
is a conformal diffeomorphism, with inverse 
\be\label{eq:invpsi}
\Phi=\Psi^{-1}\colon \R^n\setminus\{0\} \to \Sf^{n-1}\times\R\; :\;
 y\mapsto \left(\frac{y}{\|y\|},\log \|y\|\right).
 \ee 
Indeed, it is easily checked that the conformal factors of $\Psi$ and $\Phi$ are,
respectively,  $\varphi(x,t)=e^t$ and $\phi(y)=\|y\|^{-1}$.
Moreover, the unit vector field $\frac{\partial}{\partial t} \in \mathfrak X(\Sf^{n-1}\times\R)$ 
tangent to the factor $\R$ and the radial vector field ${\cal R} \in \mathfrak X(\R^n\setminus\{0\})$
are $\Psi$-related, that is, for all $(x,t)\in \Sf^{n-1}\times\R$ we have
\be\label{eq:psirelat}
 \Psi_*(x,t)\frac{\d}{\d t}(x,t) = {\cal R}(\Psi(x,t)).
 \ee
 
 A description of the submanifolds of $\mathbb R^n \setminus \{0\}$ having the constant ratio 
property with respect to $\mathcal R$ was given in \cite{byc2}. An alternative description is as follows. 

\begin{theorem} \label{thm:CRradial}  Let $\phi\colon  M^m\to \Sf^{n-1}$, $m\geq 2$, be an 
isometric immersion of a product manifold 
$M^m = J \times N^{m-1}$, with $J \subset \R$
an open interval, endowed with the polar metric~(\ref{eq:polarmet2}). Then the map
$f\colon  M^m \to \R^{n}\setminus \{0\}$ given by
\[
f(s,x) = e^{A s} \phi(s,x), \;\;A\neq 0,
\]
defines an immersion with the constant ratio property 
with respect to the radial vector field $\mathcal R$. 

Conversely, if $f\colon M^m\to\R^n\setminus\{0\}$ is an isometric immersion
with the constant ratio property with respect to $\mathcal R$, then either
$f(M)$ is contained in $\Sf^{n-1}$, or $f(M)$ is an open subset of a cone
over a submanifold $N^{m-1}$ of $\Sf^{n-1}$, or $f$ is given locally as above
(globally, if the integral curves of ${\cal R}_f^T$ are defined on $\R$).
\end{theorem} 

\proof By part $(iii)$ of Lemma \ref{le:basic}, an isometric immersion 
$f\colon  M^m \to \R^{n}\setminus\{0\}$ has the constant ratio  property with 
respect to $\mathcal R$ if and only if $f=\Psi \circ \hat f$ for some 
isometric immersion 
$\hat f\colon  M^m \to \Sf^{n-1}\times \R$ that has 
the constant ratio property with respect to $\frac{\partial}{\partial t}$, 
where $\Psi$ is given by \eqref{eq:psiradial}.
The statement then follows from  Theorem \ref{thm:CRmain}.
Notice that $f(M)$ being contained in $\Sf^{n-1}$ is equivalent to $\hat f(M)$
being contained 
in a horizontal slice of $\Sf^{n-1}\times \R$, whereas $f(M)$ being an open subset
of a cone over a submanifold $N^{m-1}$ of $\Sf^{n-1}$ is equivalent to 
$\hat f(M)$ being an open subset of a vertical cylinder $\hat N^{m-1}\times \R$.
 \vspace{2ex}\qed
 
 Notice that the cases in which $f(M)$ is contained in $\Sf^{n-1}$ or
 is an open subset of a cone
over a submanifold $N^{m-1}$ of $\Sf^{n-1}$ correspond, respectively, to the cases in which
 ${\cal R}_f^T$ or ${\cal R}_f^\perp$ vanishes identically. 
 
 Euclidean hypersurfaces of any dimension $n$ that have the constant ratio property with 
respect to the radial vector field $\mathcal R$ have been described
in \cite{byc1} (an alternative description was given  in \cite{m} for $n=2$ and in \cite{y} 
for arbitrary $n$). The next consequence of Theorem \ref{thm:CRradial}
and Lemma \ref{le:element} is essentially the description in \cite{y}. 

\begin{corollary}\label{cor:constantratiohyp} 
Let  $\phi_0\colon  N^{n-2} \to \Sf^{n-1} \subset \R^{n}$ be a hypersurface 
and let $\phi_s: N^{n-2} \to \Sf^{n-1}$ be the family of parallel 
hypersurfaces to $\phi_0$, indexed on the open interval $J \subset \R$. 
Then the map 
$f\colon M^{n-1} := N^{n-2}\times J \to \R^{n} \setminus \{0\}$, given by 
\be\label{eq:f2}
f(x,s) = e^{As} \phi_s(x), \;\;A\neq 0,
\ee
defines, at regular points, a hypersurface
that has the constant ratio property with respect to the radial vector field 
$\mathcal R$. 

Conversely, if $f\colon M^{n-1} \to \R^n$ is a hypersurface with the constant 
ratio property with respect to the radial vector field $\mathcal R$, then
either $f(M)$ is an open subset of $\Sf^{n-1}\subset\R^n$ or of a cone
over a hypersurface of $\Sf^{n-1}$, or $f$ is locally 
given as above.
\end{corollary}

Notice that the $s$-coordinate curves of the hypersurfaces \eqref{eq:f2} 
are logarithmic spirals, the plane curves that have the constant 
ratio property with respect to the radial vector field $\mathcal R$ on 
$\R^2 \setminus \{0\}$.
 
 \subsection{$T$-constant and $N$-constant Euclidean submanifolds}

  Related classes of isometric immersions $f\colon M^m\to \R^n$ are those for 
which either  $\|{\cal R}_f^T\|$ or $\|{\cal R}_f^\perp\|$ is constant on $M^m$.
These were called in \cite{byc3} $T$-submanifolds and $N$-submanifolds, respectively, 
where a description of them was given.
An alternative description of those submanifolds can be
derived as follows with our methods.

\begin{theorem} \label{thm:TNconst}  Let $\phi\colon  M^m\to \Sf^{n-1}$, $m\geq 2$, be an 
isometric immersion of a product manifold 
$M^m = J \times N^{m-1}$, with $J \subset \R$
an open interval, endowed with the  polar metric~(\ref{eq:polarmet2}).
Define $f\colon  M^m \to \R^{n}\setminus \{0\}$ by
\be\label{eq:tctemap}
f(s,x) = \rho(s) \phi(s,x),
\ee
with 
\be\label{eq:h} \rho(s)= \sec (s+C), \; \;C\in \R,
\ee 
(respectively,
\be\label{eq:h2} \rho(s)=\sqrt{1+G^2(s+C)}, \; \;C\in \R,
\ee
where $G\colon \R\to \R$ is the inverse function of 
$F\colon \R\to \R$, $F(x)=x-\arctan x$). Then the restriction of $f$
to the subset of its regular points defines an  immersion such that 
$\|{\cal R}_f^\perp\|$ (respectively,  $\|{\cal R}_f^T\|$) has 
unit length on $M^m$. 

Conversely, if $f\colon M^m\to\R^n\setminus\{0\}$ is an isometric immersion
such that $\|{\cal R}_f^\perp\|$ (respectively,  $\|{\cal R}_f^T\|$) has a
constant value $K$  
on $M^m$, then either $f(M)$ is an open subset of a cone over a submanifold 
of $\Sf^{n-1}$ (respectively, $f(M)$ is contained in $\Sf^{n-1}$) if $K=0$, 
or, if otherwise,  $f$ is locally 
(globally, if the integral curves of ${\cal R}_f^T$ are defined on $\R$) 
the composition of a map as above with a homothety of $\R^n$ of ratio $K$
($X\in \R^n\mapsto KX$).
\end{theorem} 

First we prove the following lemma.

\begin{lemma} \label{le:TNconst}  Let 
 $f\colon  M^m \to \R^{n}\setminus \{0\}$ be an isometric immersion
 and let $\hat f=\Phi\circ f$, where $\Phi\colon \R^{n}\setminus \R\to \Sf^{n-1}\times \R$
is the conformal diffeomorphism given by (\ref{eq:invpsi}).  Then  
\be\label{eq:relatnorm}
\|{\cal R}_f^\perp\|_{f}= e^h\|\hat{\cal Z}_{\hat f}^\perp\|_{\hat f}\;\;\mbox{and}\;\;
\|{\cal R}_f^T\|_f=e^h\|\hat{\cal Z}_{\hat f}^T\|_{\hat f}
\ee
where $\hat{\cal Z}=\frac{\d}{\d t}$ and $h=\<\hat f, \hat{\cal Z}\>$ is the height function 
of $\hat f$ with respect to $\hat{\cal Z}$.  Moreover, the assertions
\begin{itemize} 
\item[(i)]  $\|{\cal R}_f^\perp\|$ is constant 
on $M^m$;
\item[(ii)] 
$\nabla_{\hat{\cal Z}_{\hat f}^T}\hat{\cal Z}_{\hat f}^T=
(1-\|\hat{\cal Z}_{\hat f}^T\|^2)\hat{\cal Z}_{\hat f}^T$;
\item[(iii)]  $A^{\hat f}_{\hat{\cal Z}_{\hat f}^\perp} \hat{\cal Z}_{\hat f}^T=
(1-\|\hat{\cal Z}_{\hat f}^T\|^2)\hat{\cal Z}_{\hat f}^T$
\end{itemize}
are equivalent, and the same holds for  the assertions 
\begin{itemize} 
\item[(i')]  $\|{\cal R}_f^T\|$ is constant 
on $M^m$;
\item[(ii')]
$\nabla_{\hat{\cal Z}_{\hat f}^T}\hat{\cal Z}_{\hat f}^T
=-\|\hat{\cal Z}_{\hat f}^T\|^2\hat{\cal Z}_{\hat f}^T$.
\end{itemize}
\end{lemma} 
\proof For any $x\in M^m$ we have
\begin{eqnarray*} \< \hat{\cal Z}_{\hat f}^\perp(x), \hat{\cal Z}_{\hat f}^\perp(x)\>&=&
 \<\Phi_*(f(x)){\cal R}_f^\perp(x), \Phi_*(f(x)){\cal R}_f^\perp(x)\>\vspace{1ex}\\
 &=&\phi^2(f(x))\<{\cal R}_f^\perp(x), {\cal R}_f^\perp(x)\>\vspace{1ex}\\
&=&\frac{\<{\cal R}_f^\perp(x), {\cal R}_f^\perp(x)\>}{\|f(x)\|^2}\vspace{1ex}\\
&=&e^{-2\<\hat f(x), \hat{\cal Z}(\hat f(x))\>}\<{\cal R}_f^\perp(x), {\cal R}_f^\perp(x)\>
\end{eqnarray*}
 and 
\begin{eqnarray*} \< \hat{\cal Z}_{\hat f}^T(x), \hat{\cal Z}_{\hat f}^T(x)\>_{\hat f}
 &=&\phi^2(f(x))\<{\cal R}_f^T(x), {\cal R}_f^T(x)\>_f\vspace{1ex}\\
&=&\frac{\<{\cal R}_f^T(x), {\cal R}_f^T(x)\>_f}{\|f(x)\|^2}\vspace{1ex}\\
&=&e^{-2\<\hat f(x), \hat{\cal Z}(\hat f(x))\>}\<{\cal R}_f^T(x), {\cal R}_f^T(x)\>_f.
\end{eqnarray*}

   By the first formula in (\ref{eq:relatnorm}),  $\|{\cal R}_f^\perp\|$ is constant 
on $M^m$, say, $\|{\cal R}_f^\perp(x)\|=K\in \R$ for all $x\in M^m$, if and only if
\[ \< \hat{\cal Z}_{\hat f}^\perp(x), \hat{\cal Z}_{\hat f}^\perp(x)\>=
 K^2e^{-2\<\hat f(x), \hat{\cal Z}(\hat f(x))\>}
\]
for all $x\in M^m$. Differentiating with respect to 
$X\in \mathfrak{X}(M)$ implies that this is equivalent to
\begin{eqnarray*} \<\nabla_X^\perp \hat{\cal Z}_{\hat f}^\perp, \hat{\cal Z}_{\hat f}^\perp\>&=&
-\<\hat{\cal Z}_{\hat f}^\perp, \hat{\cal Z}_{\hat f}^\perp\>\<\hat{\cal Z}_{\hat f}^T, X\>
\end{eqnarray*}
for all $X\in \mathfrak{X}(M)$.
 Using (\ref{eq:Derivadaeta}),
the preceding equation can be written as that in item $(iii)$, 
which by (\ref{eq:NablaT}) is equivalent to the formula in item $(ii)$.

  Finally, by the second formula in (\ref{eq:relatnorm}),  $\|{\cal R}_f^T\|$ is constant 
on $M^m$, say, $\|{\cal R}_f^T(x)\|=K\in \R$ for all $x\in M^m$, if and only if
\[ \< \hat{\cal Z}_{\hat f}^T(x), \hat{\cal Z}_{\hat f}^T(x)\>=
 K^2e^{-2\<\hat f(x), \hat{\cal Z}(\hat f(x))\>}
\]
for all $x\in M^m$. 
Differentiating with respect to 
$X\in \mathfrak{X}(M)$ implies that this is equivalent to
\begin{eqnarray*} \<\nabla_X \hat{\cal Z}_{\hat f}^T, \hat{\cal Z}_{\hat f}^T\>&=&
-\<\hat{\cal Z}_{\hat f}^T, \hat{\cal Z}_{\hat f}^T\>\<\hat{\cal Z}_{\hat f}^T, X\>
\end{eqnarray*}
for all $X\in \mathfrak{X}(M)$. 
 Using that $\hat{\cal Z}_{\hat f}^T$ is a gradient vector field implies that
the preceding equation is equivalent to that in item $(ii')$. \qed\vspace{2ex}

\noindent \emph{Proof of Theorem \ref{thm:TNconst}:} Let $\hat f=\Phi\circ f$, where 
$\Phi\colon \R^{n}\setminus \{0\}\to \Sf^{n-1}\times \R$
is the conformal diffeomorphism given by (\ref{eq:invpsi}). Then
$$
\hat f(s,x)=(\phi(s,x), h(s))
$$
for all $(s,x)\in M^m$. Arguing as in the beginning of the proof of Theorem \ref{thm:CRmain}, 
we see that, if $\hat{\cal Z}=\frac{\d}{\d t}\in \mathfrak{X}(\Sf^{n-1}\times \R)$, then 
$\hat{\cal Z}_{\hat f}^T=\lambda \frac{\d}{\d s}$ for some $\lambda \in C^{\infty}(M)$. 
Moreover, 
$$
h'(s)=\frac{\partial}{\partial s}\<\hat f,\hat{\cal Z}\>
 =  \<\hat f_*\frac{\partial}{\partial s}, \hat{\cal Z}\> 
 = \< \frac{\partial}{\partial s}, \hat{\cal Z}^T_{\hat f} \>_{\hat f} 
 = \lambda \< \frac{\partial}{\partial s}, \frac{\partial}{\partial s} \>_{\hat f} 
 = (1+(h'(s))^2)\lambda,
$$
hence 
$$
\|\hat{\cal Z}_{\hat f}^T\|^2=\frac{(h'(s))^2}{1+(h'(s))^2}\,\,\mbox{and}\,\, 
\|\hat{\cal Z}_{\hat f}^\perp\|^2=\frac{1}{1+(h'(s))^2}
$$
which implies that $e^h\|\hat{\cal Z}_{\hat f}^\perp\|$ (respectively, $e^h\|\hat{\cal Z}_{\hat f}^T\|$) 
is constant if $h$ is given by (\ref{eq:h}) (respectively, (\ref{eq:h2})).  Thus the statement 
follows from  (\ref{eq:relatnorm}).

Now we prove the converse. If either  $\|{\cal R}_f^\perp\|$ or $\|{\cal R}_f^T\|$  is constant 
on $M^m$, then both the equations in items $(ii)$ and $(ii')$ 
imply that the 
integral curves of $\hat{\cal Z}_{\hat f}^T/\|\hat{\cal Z}_{\hat f}^T\|$ are geodesics.
Moreover, using that $\hat{\cal Z}_{\hat f}^T$ is the gradient of
the height function $h=\<\hat f,\hat {\cal Z}\>$, and hence that the orthogonal distribution
$\hat{\cal Z}_{\hat f}^T$ is integrable,  
we see that $M^m$ is locally diffeomorphic to a product manifold $M^m=I\times N^{m-1}$, and that 
 $$f(x,s)=(\phi(x,s), h(s)$$
 for some immersion  $\phi\colon M^m\to \Sf^{m-1}$  and some $h\in C^{\infty}(I)$. 
 Moreover, if for each $s\in I$ we denote by $g_s$ the metric induced by the map 
 $$x\in M^m \mapsto  \phi(x,s)\in  \Sf^{m-1},$$
 then the fact that the $s$-coordinate curves $s\mapsto (s,x)$ (the integral curves of 
 $\hat{\cal Z}_{\hat f}^T$) are reparametrizations of geodesics says that the metric induced by $f$ is
 $$d\tilde \sigma=(1+(h'(s))^2)ds^2+g_s.$$
 Therefore, as in the proof of the direct statement, the assumption that 
 $\|{\cal R}_f^\perp\|$ or  $\|{\cal R}_f^T\|$  has a
constant value $K$  
on $M^m$ translates, respectively, into the ODEs 
$$ \frac{e^{2h}}{1+(h')^2}=K\;\;\mbox{and}\;\; \frac{e^{2h}(h')^2}{1+(h')^2}=K$$
for $h$, whose solutions are easily checked to be given by 
 (\ref{eq:h}) and (\ref{eq:h2}), respectively. \qed
 
 \section{Principal direction property for radial vector fields}

Corollary \ref{prop:charac2} takes the following simpler form for
radial vector fields.

\begin{corollary}\label{prop:charac} 
An isometric immersion $f\colon M^m\to\R^n\setminus\{0\}$ has the principal 
direction property with respect to the radial vector field ${\cal R}$  if 
and only if ${\cal R}^{\perp}_f$ is parallel along $\{{\cal R}^T_f\}^\perp$ with respect
to the normal connection.
\end{corollary}
\begin{proof}
Denoting by $\tilde \nabla$ the Euclidean connection and by $\nabla$ the
Levi-Civita connection of $M^n$, the Gauss and Weingarten formulas yield
\begin{eqnarray*}
f_*X & = & \tilde \nabla_X (\mathcal R\circ f) 
= \tilde \nabla_X(f_*\mathcal R^T_f + \mathcal R^{\perp}_f) \\
   & = &  f_*\nabla_X \mathcal R^T_f + \alpha_f(X,\mathcal R^T_f) 
- f_*A_{\mathcal R^{\perp}_f}X + \nabla_X^{\perp} \mathcal R^{\perp}_f.
\end{eqnarray*}
The normal component of this equation reads 
$\nabla_X^{\perp} \mathcal R^{\perp}_f = - \alpha_f(X,\mathcal R^T_f)$, 
which implies the statement.
\end{proof}

In the case of hypersurfaces,   Corollary \ref{prop:charac} reads
as follows.

\begin{corollary}\label{prop:charachyp} 
A hypersurface $f:M^{n-1}\to\R^{n}\setminus\{0\}$ has the principal
direction property with respect to the radial vector field ${\cal R}$ if
and only if $\|{\cal R}^\perp_f\|$ is constant along $\{\mathcal R^T_f\}^\perp$. 
\end{corollary}

   By the observation after Lemma \ref{le:basic}, every hypersurface of
$\R^{n}\setminus\{0\}$ that has the constant ratio property with respect
to $\mathcal R$ also has the principal direction property with respect to
$\mathcal R$. Indeed, the radial vector field $\mathcal R$ is $\Psi$-related
the parallel unit vector field $\frac{\partial}{\partial t} \in \mathfrak X(\Sf^{n-1}\times\R)$ 
tangent to the factor $\R$, where $\Psi\colon \Sf^{n-1}\times\R \to \R^n \setminus\{0\}$
is the conformal diffeomorphism given by (\ref{eq:psiradial}). This also follows
from the preceding corollary by noticing that, since $\|\mathcal R\circ f\|$
is constant along $\{\mathcal R^T_f\}^\perp$, then 
$\|{\cal R}^\perp_f\|$ is also constant along $\{\mathcal R^T_f\}^\perp$
if $f$ is a hypersurface with the constant ratio property with respect
to $\mathcal R$.\vspace{1ex}

The next result gives a description of all isometric immersions
$f\colon M^m \to \R^{n}\setminus\{0\}$, $m\geq 2$, that have the principal
 direction property with respect to ${\cal R}$.

\begin{theorem} \label{thm:PDradial} 
Let $\phi \colon\, N^{m-1}\to \Sf^{n-1} \subset \R^{n}$ be an isometric 
immersion along which there exists an orthonormal set of normal 
vector fields $\xi_1, \ldots, \xi_k$ that are parallel in the normal bundle.
Let $\gamma:I\to\Sf^k\times\R\subset \R^{k+2}$, with
$\gamma(s)= (\gamma_1(s),\ldots,\gamma_{k+1}(s),\gamma_{k+2}(s))$,
be a smooth regular curve such that $\gamma_{k+2}$ has non-vanishing
derivative. Then the restriction of the map $f\colon N^{m-1} \times I \to \R^{n}\setminus\{0\}$,
given by
\begin{equation} \label{eq:fPDradial}
f(x,s) = e^{\gamma_{k+2}(s)} \left( \sum_{i=1}^k \gamma_i(s)\xi_i(x) 
+ \gamma_{k+1}(s)\phi(x) \right),
\end{equation}
to the subset of its  regular points defines an immersion
with the principal direction property with respect to ${\cal R}$. 

Conversely, if $f\colon M^m\to\R^n\setminus\{0\}$ is an isometric immersion
with the principal direction property with respect to $\mathcal R$, then either
$f(M)$ is an open subset of a cone
over a submanifold $N^{m-1}$ of $\Sf^{n-1}$, or $f$ is given locally as above
(globally, if the geodesic integral 
curves of ${\cal R}_{f}^T$ are defined on $\R$). 
\end{theorem}

\begin{proof}
By part $(iv)$ of Lemma \ref{le:basic}, an isometric immersion 
$f\colon M^m\to\R^{n}\setminus\{0\}$ has the principal direction property 
with respect to  ${\cal R}$ if and only if $f=\Psi \circ \hat f$ 
for some isometric immersion $\hat f\colon M^m\to \Sf^{n-1}\times \R$ with 
the principal direction property with respect to $\frac{\partial}{\partial t}$. 
The statement then follows from  Theorem \ref{thm:PDmain}.
\end{proof}

\begin{corollary} \label{cor:canonicalfieldhyp} 
Let  $\phi_0\colon N^{n-2} \to \Sf^{n-1} \subset \R^{n}$ be a hypersurface and 
let $\phi_s\colon N^{n-2} \to \Sf^{n-1}$ be the family of parallel 
hypersurfaces to $\phi_0$, indexed on an open interval $J \subset \R$. 
Then the restriction of the map 
$f\colon M^{n-1} \colon= N^{n-2} \times J \to \R^{n} \setminus \{0\}$, given by 
\be\label{eq:f2b}
f(x,s)=e^{a(s)}\phi_s(x),
\ee
to the subset of its regular points defines a hypersurface  with  the principal direction 
property with respect to ${\cal R}$. 

Conversely, if $f\colon M^{n-1}\to\R^n\setminus\{0\}$, $n\geq 3$, is an 
isometric immersion
with  the  principal direction property with 
respect to $\mathcal R$, then either
$f(M)$ is an open subset of a cone
over a hypersurface of $\Sf^{n-1}$, or $f$ is given locally as above
(globally, if the  integral 
curves of ${\cal R}_{f}^T$ are defined on $\R$). 
\end{corollary}

  Notice that $f(M)$ is foliated by the plane spirals
  $$s\mapsto f(x,s)=e^{a(s)}(\cos s \, \phi(x) + \sin s \, N(x)),$$
  and the orthogonal hypersurfaces $x\mapsto f(x,s)$ are homothetical to the parallel
  hypersurface $\phi_s$ of $\phi$. 

For $n=3$, Corollary \ref{cor:canonicalfieldhyp} reduces to the main 
theorem of \cite{mf}, where surfaces in $\R^3\setminus\{0\}$ with the 
principal direction property with respect to ${\cal R}$ were called 
\emph{generalized constant ratio} surfaces, in view of the fact that 
any constant ratio surface in $\R^3$ has this property, as pointed out after
Corollary \ref{prop:charachyp}. This terminology is not appropriate
for Euclidean submanifolds of codimension greater than one, for in this case
the constant ratio property with respect to ${\cal R}$ no longer implies 
the principal direction property with respect to ${\cal R}$. 

The following result classifies isometric immersions into 
$\R^{n}\setminus \{0\}$ that have both the constant ratio and the 
principal direction properties with respect to $\mathcal R$.

\begin{corollary} 
Let $f\colon  M^m= N^{m-1} \times I \to \R^n\setminus\{0\}$ be given by 
\eqref{eq:fPDradial} for a 
regular curve 
$\gamma\colon I\subset\R\to\Sf^k\times\R$, 
$\gamma(s) = (\gamma_1(s),\ldots, \gamma_{k+1}(s), \gamma_{k+2}(s))$, such that 
$\bar \gamma\colon I\to \Sf^k\subset \R^{k+1}$, 
$\bar \gamma(s)= (\gamma_1(s),\ldots, \gamma_{k+1}(s))$, is a unit-speed curve and 
$\gamma_{k+2}(s))=As$ for some $A\neq 0$. 
 Then 
the restriction of $f$ to the subset of its regular points is an
immersion with both the principal direction and constant 
ratio properties with respect to $\mathcal R$.

Conversely, if $f\colon M^m\to\R^n\setminus\{0\}$, $m\geq 2$, is an 
isometric immersion
with both the constant ratio and principal direction properties with 
respect to $\mathcal R$, then either
$f(M)$ is an open subset of a cone
over a submanifold $N^{m-1}$ of $\Sf^{n-1}$, or $f$ is given locally as above
(globally, if the  integral 
curves of ${\cal R}_{f}^T$ are defined on $\R$). 
\end{corollary}

\begin{proof} By parts $(iii)$ and $(iv)$ of Lemma \ref{le:basic}, any
isometric immersion $f:M^m\to \R^{n}\setminus\{0\}$ that has both the constant 
ratio and principal direction properties with respect to ${\cal R}$ is given by 
$f=\Psi\circ \hat f$ for some isometric immersion $\hat f\colon M^{m}\to \Sf^{n-1}\times \R$
that has both the constant ratio and principal direction properties with respect to
$\frac{\partial}{\partial t}$, where $\Psi$ is given by \eqref{eq:psiradial}. 
The statement then follows from Corollary \ref{cor:CRandPD}. 
\end{proof}

\subsection{Euclidean submanifolds with parallel ${\mathcal R}^{\perp}_f$}

    In view of Corollary \ref{prop:charac}, it is natural to ask  which 
isometric immersions $f\colon M^m\to\R^n\setminus\{0\}$ have the property that
${\cal R}^{\perp}_f$ is parallel on $M^m$ with respect
to the normal connection. The next result classifies such submanifolds.

\begin{theorem} \label{cor:parallelPD} 
Let $f\colon  M^m= N^{m-1} \times I \to \R^n\setminus\{0\}$ be given by 
\eqref{eq:fPDradial} for a 
regular curve 
$\gamma\colon I\subset\R\to\Sf^k\times\R$, 
$\gamma(s) = (\gamma_1(s),\ldots, \gamma_{k+1}(s), \gamma_{k+2}(s))$, such that 
$\bar \gamma\colon I\to \Sf^k\subset \R^{k+1}$, 
$\bar \gamma(s)= (\gamma_1(s),\ldots, \gamma_{k+1}(s))$, is a geodesic 
and $\gamma_{k+2}(s)=\log(\sec (s+C))$ for some $C\in \R$ and all $s\in I$. 
 Then the restriction of to the subset of its regular points is an immersion 
 with the property that ${\cal R}_f^\perp$ is parallel in the normal connection.

Conversely, if $f\colon M^m\to\R^n\setminus\{0\}$ is an isometric immersion
with the property that ${\cal R}_f^\perp$ is parallel in the normal connection, then either
$f(M)$ is an open subset of a cone
over a submanifold $N^{m-1}$ of $\Sf^{n-1}$, or $f$ is given locally as above
(globally, if the  integral 
curves of ${\cal R}_{f}^T$ are defined on $\R$), up to a homothety of $\R^n$. 
\end{theorem}

First we prove the following lemma

\begin{lemma} \label{le:parallelperp}  Let 
 $f\colon  M^m \to \R^{n}\setminus \{0\}$ be an isometric immersion
 and let $\hat f=\Phi\circ f$, where $\Phi\colon \R^{n}\setminus \{0\} \to \Sf^{n-1}\times \R$
is the conformal diffeomorphism given by (\ref{eq:invpsi}).  Then ${\cal R}_f^\perp$ is parallel
on $M^m$ with respect to the normal connection if and only if 
 \be\label{eq:parallelperp}
 A^{\hat f}_{\xi} \hat{\cal Z}_{\hat f}^T=\<\hat{\cal Z}_{\hat f}^\perp, \xi\>\hat{\cal Z}_{\hat f}^T
 \ee
 for all $\xi\in \Gamma(N_{\hat f}M)$, where $\hat{\cal Z}=\frac{\d}{\d t}$.
\end{lemma} 
\proof  By (\ref{eq:nderiv}),  the normal derivatives $\hat\nabla_X^\perp \hat{\cal Z}^\perp_{\hat f}$ and 
$\nabla_X^\perp {\cal R}^\perp_{f}$ are related by
$$
\hat\nabla_X^\perp \hat{\cal Z}^\perp_{\hat f}=\Phi_* \,
\nabla_X^\perp {\cal R}^\perp_{f} 
+ \frac{X(\phi\circ f)}{\phi\circ f}\hat{\cal Z}^\perp_{\hat f}
$$
for any $X\in \mathfrak{X}(M)$, where $\phi$ is the conformal factor of $\Phi$.
Thus ${\cal R}_f^\perp$ is parallel 
on $M^m$ with respect to the  normal connection if and only if
\be\label{eq:par2}
\hat\nabla_X^\perp \hat{\cal Z}^\perp_{\hat f}= 
\frac{X(\phi\circ f)}{\phi\circ f}\hat{\cal Z}^\perp_{\hat f}
\ee
for all $X\in \mathfrak{X}(M)$.
Now, 
$\phi(y)=\frac{1}{\|y\|}$
for all $y\in \R^n\setminus \{0\}$, so 
\begin{eqnarray*}\frac{X(\phi\circ f)}{\phi\circ f}&=&-X(\log \|f\|)\\
&=&-X\<\Phi\circ f, \hat {\cal Z}\>\\
&=&-X\<\hat f, \hat {\cal Z}\>\\
&=&-\<X, \hat{\cal Z}_{\hat f}^T\>
\end{eqnarray*}
for all $X\in \mathfrak{X}(M)$, and therefore (\ref{eq:par2}) becomes
$$
\hat\nabla_X^\perp \hat{\cal Z}^\perp_{\hat f}= -\<X, \hat{\cal Z}_{\hat f}^T\>
\hat{\cal Z}^\perp_{\hat f}
$$
for all $X\in \mathfrak{X}(M)$. It follows from  (\ref{eq:Derivadaeta}) 
that the preceding equation is equivalent to (\ref{eq:parallelperp}). \vspace{2ex} \qed

\noindent \emph{Proof of Theorem \ref{cor:parallelPD}:}
Let $\hat f=\Phi\circ f$, where $\Phi\colon \R^{n}\setminus \{0\}\to \Sf^{n-1}\times \R$
is the conformal diffeomorphism given by (\ref{eq:invpsi}). Then $\hat f$ is given by
(\ref{eq:fPD}) in terms of a regular curve 
$\gamma\colon I\subset\R\to\Sf^k\times\R\subset \R^{k+2}$, 
$\gamma(s)= (\gamma_1(s),\ldots, \gamma_{k+1}(s), \gamma_{k+2}(s))$.
Thus $\hat f$ has the principal direction
property with respect to $\hat{\cal Z}=\frac{\d}{\d t}$ by Theorem \ref{thm:PDmain}.
Moreover, since $\bar \gamma\colon I\to \Sf^k\subset \R^{k+1}$, 
$\bar \gamma(s)= (\gamma_1(s),\ldots, \gamma_{k+1}(s))$, is a geodesic
of $\Sf^k$, it follows from (\ref{eq:alphapt3}) that $A^{\hat f}_{\xi} \hat{\cal Z}_{\hat f}^T=0$
for all $\xi\in \Gamma(N_{\hat f}M)$ with $\<\xi,\hat{\cal Z}_{\hat f}^\perp\>=0$. 
Finally, since the height function 
$\<\hat f, \hat {\cal Z}\>=\gamma_{k+2}(s)=\log(\sec (s+C))$ for some  $C\in \R$ and all $s\in I$,
it follows from Theorem \ref{thm:TNconst} that ${\cal R}_{f}^\perp$
has constant length. By item $(iii)$ of Lemma \ref{le:TNconst}, formula (\ref{eq:parallelperp})
also holds for $\xi= \hat{\cal Z}_{\hat f}^\perp$.

Conversely, if $f\colon M^m\to\R^n\setminus\{0\}$ is an isometric immersion
with the property that ${\cal R}_f^\perp$ is parallel in the normal connection,
then $f$ has the principal direction property with respect to ${\cal R}$ by 
Corollary \ref{prop:charac}. Therefore $\hat f=\Phi\circ f$, 
where $\Phi\colon \R^{n}\setminus \{0\}\to \Sf^{n-1}\times \R$
is the conformal diffeomorphism given by (\ref{eq:invpsi}), has 
the principal direction property with respect to $\hat{\cal Z}=\frac{\d}{\d t}$
by part $(iv)$ of Lemma \ref{le:basic}, and hence it is given  by
(\ref{eq:fPD}) in terms of a regular curve 
$\gamma\colon I\subset\R\to\Sf^k\times\R\subset \R^{k+2}$, 
$\gamma(s) = (\gamma_1(s),\ldots, \gamma_{k+1}(s), \gamma_{k+2}(s))$,
such that $\gamma'_{k+2}$
has nonvanishing derivative by Theorem \ref{thm:PDmain}.
Since $A^{\hat f}_{\xi} \hat{\cal Z}_{\hat f}^T=0$
for all $\xi\in \Gamma(N_{\hat f}M)$ with $\<\xi,\hat{\cal Z}_{\hat f}^\perp\>=0$
by (\ref{eq:parallelperp}), it follows from (\ref{eq:alphapt3}) that
$\bar \gamma\colon I\to \Sf^k\subset \R^{k+1}$, 
$\bar \gamma(s)= (\gamma_1(s),\ldots, \gamma_{k+1}(s))$, is a geodesic
of $\Sf^k$. Finally, since ${\cal R}_{f}^\perp$ has constant length,  
it follows from Theorem \ref{thm:TNconst} that $\gamma_{k+2}(s)=\log(K\sec (s+C))$ 
for some $K>0$, $C\in \R$ and all $s\in I$. \qed

\section{The case of  Killing vector fields}

Let $x_1, \ldots, x_{n+1}$ be the standard coordinates in $\R^{n+1}$ and let 
$\d_{x_i}$ denote a unit vector field tangent to the $x_i$-coordinate curve, $1\leq i\leq n+1$.
The Lie algebra  of Killing vector fields in $\R^{n+1}$ has dimension  $\frac{1}{2}(n+1)(n+2)$
and is generated by the constant vector fields
\be\label{eq:ckvf}
\d_{x_i}, \;\;1\leq i\leq n+1,
\ee
and the vector fields
\be\label{eq:rkvf}
 {\cal K}_{ij}=x_i\d_{x_j}- x_j\d_{x_i},\;\; 1\leq i\neq j\leq n+1,
 \ee
generating rotations around the linear subspaces $\R^{n-1}$ of $\R^{n+1}$
given by  $x_{i} = 0 = x_{j}$. \vspace{1ex}

In this section we describe all isometric immersions that have either the 
constant ratio or the principal 
direction property with respect to ${\cal K}_{ij}$. We work with, 
say, ${\cal K}_{n,n+1}$, which we denote simply by ${\cal K}$. 

  For this purpose, we make use of the conformal diffeomorphism between
$\Hy^n \times \Sf^1$ and $\mathbb R^{n+1} \setminus \mathbb R^{n-1}$ given as follows. 
Let $e_0, e_1, \ldots, e_{n-1},e_n$ be a pseudo-orthonormal basis of the 
Lorentzian space $\R^{n+1}_1$ satisfying 
\be\label{eq:pseudo}
\<e_0,e_0\>=0=\<e_n, e_n\>,\;\;\<e_0, e_n\>=-1/2\;\;\mbox{and}\;\;\<e_i,e_j\>
=\delta_{ij},\;\;1\leq i\leq n-1, \;\;0\leq j\leq n.
\ee
 Then the map 
$
\Psi \colon \Hy^n\times \Sf^{1}\subset \R^{n+1}_1 
\times \R^{2} \to \R^{n+1}\setminus \R^{n-1}
$
given by
\[
\Psi(x_0 e_0 + \ldots + x_{n} e_n, (y_1, y_2)) = \frac{1}{x_0}(x_1, 
\ldots, x_{n-1}, y_1, y_2)
\]
is a conformal diffeomorphism with conformal factor  
$$ \varphi(x_0 e_0 + \ldots + x_{n} e_n, (y_1, y_2)) = \frac{1}{x_0}, $$ 
whose inverse $\Psi^{-1}: \R^{n+1}\setminus \R^{n-1} \to 
\Hy^n \times \Sf^{1} \subset \R^{n+1}_1 \times \R^{2}$ is given by
$$ \Psi^{-1}(y_1, \ldots, y_{n+1}) = 
\frac{1}{\sqrt{y_n^2+y_{n+1}^2}} \left( e_0 + \sum_{i=1}^{n-1} y_i e_i 
+ \left( \sum_{i=1}^{n+1} y_i^2 \right) e_n, (y_n, y_{n+1}) \right). $$
Notice that the metric induced by the restriction of $\Psi^{-1}$ to each 
half-space of a hyperplane of $\R^{n+1}$ containing $\R^{n-1}$ is the standard
hyperbolic metric of the half-space model of $\Hy^n$. In other words, the restriction
of the conformal diffeomorphism $
\Psi \colon \Hy^n\times \Sf^{1}\subset \R^{n+1}_1 
\times \R^{2} \to \R^{n+1}\setminus \R^{n-1}
$
to each slice $\Hy^n\times \{z\}\subset \Hy^n\times \Sf^{1}$ gives an isometry 
of the hyperboloidal model of $\Hy^n$ onto its half-space model.

Composing $\Psi$ with the isometric covering map 
$$
\pi\colon \Hy^n\times \R\to \Hy^n\times \Sf^1: (x,t) \mapsto (x, (\cos t, \sin t))
$$ 
produces a conformal covering map 
$\tilde \Psi \colon \Hy^n\times \R\to \R^{n+1}\setminus \R^{n-1}$
given by
\be\label{tpsi}
\tilde\Psi(x_0 e_0 + \ldots + x_n e_n, t) = 
\frac{1}{x_0}(x_1, \ldots, x_{n-1}, \cos t, \sin t).
\ee
The reason $\tilde \Psi$ is useful for our purposes is that the unit vector field 
$\frac{\partial}{\partial t} \in \mathfrak X(\Hy^n \times \R)$ is 
$\tilde\Psi$-related to the Killing vector field 
$\mathcal K \in \mathfrak X(\R^{n+1}\setminus \R^{n-1})$, namely, 
\[
\tilde\Psi_*(x,t)\frac{\d}{\d t}(x,t) ={\cal K}(\tilde\Psi(x,t))
\]
for all $x=(x_0 e_0 + \ldots + x_n e_n,t) \in \Hy^n \times \R$. 

\subsection{The constant ratio property with respect to $\mathcal K$}

Submanifolds of $\R^{n+1}\setminus \R^{n-1}$ having the constant ratio property 
with respect to $\mathcal K$ can be classified as follows by using 
 part $(iii)$ of Lemma \ref{le:basic} and the conformal covering map  (\ref{tpsi}).

\begin{theorem}\label{thm:CRkilling}
Let $\phi\colon M^m\to\Hy^n\subset\R^{n+1}_1$, $m\geq 2$, be an isometric
immersion of a product manifold $M^m={J}\times N^{m-1}$ endowed with
a polar metric. Write 
$\phi(s,x)=\sum_{j=0}^n\phi_j(s,x)e_j$,
where  $e_0, \ldots, e_{n}$ is a pseudo-orthonormal basis of $\R^{n+1}_1$ as 
in (\ref{eq:pseudo}).  Then  the map
$f\colon  M^m \to \R^{n+1} \setminus \R^{n-1}$ given by
\be\label{eq:fCRkilling}
f(s,x) = \frac{1}{\phi_0(s,x)}(\phi_1(s,x), \ldots, \phi_{n-1}(s,x),\cos(A s),
\sin(A s))
\ee
defines an immersion with the constant ratio 
property with 
respect to the Killing vector field $\mathcal K$.

Conversely, if $f\colon M^m \to \R^{n+1} \setminus \R^{n-1}$, $m\geq 2$, is 
an isometric immersion that has the constant ratio property with respect 
to $\mathcal K$, then it is either a rotational submanifold having $\R^{n-1}$
as axis, or $f(M^m)$ lies in a hyperplane that contains the subspace $\R^{n-1}$,
or $f$ is locally given as  above 
(globally, if the integral curves of $\mathcal K_{f}^T$ are defined on $\R$).
\end{theorem} 
\proof Since the map $f: M^m \to \R^{n+1} \setminus \R^{n-1}$ given by
(\ref{eq:fCRkilling}) is the composition $f=\tilde \Psi\circ \hat f$ of
the conformal covering map  (\ref{tpsi}) with an isometric immersion
$\hat f\colon M^m\to \Hy^{n}\times \R$ as in Theorem \ref{thm:CRmain},
the first statement follows from part $(iii)$ of Lemma \ref{le:basic}.

The converse also follows from the converse statement of Theorem \ref{thm:CRmain}
and part $(iii)$ of Lemma \ref{le:basic} by noticing that 
rotational submanifolds having $\R^{n-1}$
as axis are precisely the images under the conformal covering map 
$\tilde \Psi \colon \Hy^n\times \R\to \R^{n+1}\setminus \R^{n-1}$ of the vertical 
cylinders in  $\Hy^n\times \R$, whereas submanifolds of $\R^{n+1}\setminus \R^{n-1}$
that lie in a hyperplane of $\R^{n+1}$ containing the subspace $\R^{n-1}$
are the images under $\tilde \Psi$ of submanifolds of $\Hy^n\times \R$
that are contained in a horizontal slice of $\Hy^n\times \R$.\vspace{1ex}\qed

     A more explicit description of  hypersurfaces 
$f\colon M^m \to \R^{n+1} \setminus \R^{n-1}$, $m\geq 2$, with the constant ratio 
property with respect to $\mathcal K$, or equivalently, whose unit normal vector field
makes a constant angle with $\mathcal K$, is as follows.

\begin{corollary} \label{cor:CRkillinghyp} 
Let $\phi\colon N^{n-1} \to \Hy^n$ be any hypersurface and let 
$\phi_s\colon N^{n-1} \to \Hy^n\subset \R^{n+1}_1$ be the family of  parallel hypersurfaces to 
$\phi$, indexed on the open interval $J \subset \R$.
Write
$\phi_s(x)=\sum_{j=0}^n\phi_j(s,x)e_j$
where  $e_0, \ldots, e_{n}$ is a pseudo-orthonormal basis of $\R^{n+1}_1$ as in (\ref{eq:pseudo}).
Then the map 
$f\colon M^n:= J \times N^{n-1} \to \R^{n+1} \setminus \R^{n-1}$,  given by
\be\label{eq:fCRkillinghyp}
f(x,s) = \frac{1}{\phi_0(s,x)}(\phi_1(s,x), \ldots, \phi_{(n-1)}(s,x),\cos(As),\sin(As)),\;\;A\neq 0,
\ee
is a hypersurface with the constant ratio property with respect to $\mathcal K$.

Conversely, if $f\colon  M^n \to \R^{n+1} \setminus \R^{n-1}$, $n\geq 2$, 
is a hypersurface that has the constant ratio property with respect to the Killing vector 
field $\mathcal K$, then either $f$ is  a rotational hypersurface  having $\R^{n-1}$
as axis, or $f(M^n)$ is an open subset of a hyperplane that contains $\R^{n-1}$, 
or it is locally given as  above 
(globally, if the integral curves of $\mathcal K_{ f}^T$ are defined on $\R$).
\end{corollary}

\begin{remark} \label{munteanu} \emph{ For $n=2$, Corollary \ref{cor:CRkillinghyp} says 
that any surface $f\colon  M^2 \to \R^{3} \setminus \R$ with the constant ratio property 
with respect to the Killing vector field $\mathcal K$ is  either a rotational surface  
having $\R$ as axis, an open subset of a plane that contains $\R$, or it is locally 
(globally, if the integral curves of $\mathcal K_{ f}^T$ are defined on $\R$) a surface 
$f\colon M^2:= J \times I \to \R^{3} \setminus \R$  given by
\be\label{eq:fCRkillingsup}
f(t,s) = \frac{1}{\gamma_0(s,t)}(\gamma_1(s,t),\cos(As),\sin(As)),\;\;A\neq 0,
\ee
where $\gamma\colon I\to \Hy^2\subset \R_1^3$, 
$\gamma(t)=\gamma_0(t) e_0+\gamma_1(t) e_1 +\gamma_2(t) e_2$, is any unit-speed curve, and 
$$\gamma(t,s)=\gamma_0(t,s) e_0+\gamma_1(t,s) e_1 +\gamma_2(t,s) e_2=\cosh s\gamma(t)+\sinh s n(t)$$
is the family of parallel curves to $\gamma$. Here  $n(t)\in T_{\gamma(t)}\Hy^2$
is a unit vector orthogonal to $\gamma'(t)$, and  $e_0, e_1, e_2$ is a a pseudo-orthonormal basis 
$e_0, e_1, e_2$ of $\R^{3}_1$ such that $\<e_0,e_0\>=0=\<e_2, e_2\>$, $\<e_0, e_2\>=-1/2$ and 
$\<e_1, e_1\>=1$.}  

\emph{Taking $\gamma\colon \R \to \Hy^2\subset \R_1^3$ as the horocycle
$$ \gamma(t)=e_0+te_1+(t^2+1)e_2,$$
then $\gamma(t,s)=e^se_0+te^se_1+(t^2e^s+e^{-s})e_2$, and the corresponding surface
$f\colon \R^2 \to \R^{3} \setminus \R$  is given by
$$
f(t,s) = \frac{1}{e^s}(te^s,\cos(As),\sin(As))=(t,e^{-s}\cos(As),e^{-s}\sin(As))\;\;A\neq 0,
$$
the cylinder over the logarithmic spiral.}

  \emph{ A remarkable example appears by starting with the curve $\gamma\colon \R \to \Hy^2\subset \R_1^3$
  given for any $\sigma\in \R$ by
$$
 \begin{array}{l}
{\displaystyle \gamma(t)=\frac{\cosh (t/\cos \sigma)}{\cos \sigma} e_0
+((t/\cos \sigma)\cosh (t/\cos \sigma)-\sinh (t/\cos \sigma))e_1}\vspace{2ex}\\
\hspace*{5ex}
+\cos \sigma(\cosh (t/\cos \sigma)+(t/\cos \sigma)^2\cosh (t/\cos \sigma)-2(t/\cos \sigma)\sinh (t/\cos \sigma))e_2.
\end{array}
$$
From ${\displaystyle \<\gamma'(t), \gamma'(t)\>=\frac{\sinh (t/\cos \sigma)}{\cos \sigma}}$ it follows that 
${\displaystyle \frac{\<\gamma'(t), e_2\>}{\|\gamma'(t)\|}=\cos \sigma}$ for all $t\in \R$, that is, 
$\gamma$ is a helix in $\R_1^3$ with  a light-like axis (see \cite{dt} for a parametrization of all 
helices in $\R_1^3$ lying in $\Hy^2\subset \R_1^3$, as well as for some of their properties).
One can check that a unit vector field $n(t)\in T_{\gamma(t)}\Hy^2$ orthogonal to $\gamma'(t)$ is
$$
\begin{array}{l}
{\displaystyle n(t)=-\frac{\sinh (t/\cos \sigma)}{\cos{\sigma}}e_0+(\cosh (t/\cos \sigma)
-(t/\cos \sigma)\sinh (t/\cos \sigma))e_1}\vspace{2ex}\\
\hspace*{5ex} 
+\cos \sigma(2(t/\cos \sigma)\cosh (t/\cos \sigma)-(1+(t/\cos \sigma)^2)\sinh (t/\cos \sigma))e_2,
\end{array}
$$
and that 
$$
{\displaystyle \gamma(t,s)=-\frac{\cosh \rho}{\cos{\sigma}}e_0+\left(\frac{\cosh \rho}{\cos{\sigma}} t
-\sinh \rho\right)e_1} +\cos \sigma\cosh \rho+\frac{\cosh \rho}{\cos \sigma}t^2-2t\sinh \rho)e_2,
$$
where ${\displaystyle \rho=\frac{t-s \sin \sigma}{\cos \sigma}}$. The corresponding surface 
$f\colon  M^2 \to \R^{3} \setminus \R$ has the constant ratio property with respect to the 
Killing vector field $\mathcal K$ and is parametrized by
$$
{\displaystyle f(t,s)=\left(\frac{\cos\sigma\cos s}{\cosh \rho}, \frac{\cos\sigma\sin s}{\cosh \rho}, 
t-\cos{\sigma}\tanh \rho \right)},
$$
which is Dini's helicoidal surface of constant negative Gauss curvature. 
Therefore, Dini's surface is the image, under the conformal covering map 
$\Phi\colon \Hy^2\times \R\to \R^3\setminus \R$, of the surface in 
$\Hy^2\times \R \subset \R_1^4$ that is generated by starting with a helix 
$\gamma\colon \R\to \Hy^2\subset \R_1^{3}\subset \R_1^4$ with a light-like axis,
taking a standard helix $\beta\colon \R\to \Hy^1\times \R\subset \R_1^3$ 
in a fixed normal space of $\gamma$ in $\R_1^4$, and then
parallel translating $\beta$ along $\gamma$ with respect to the normal connection of $\gamma$.
}

\emph{ Notice also that $t\mapsto f(t,0)$ is a parametrization of a tactrix in a half-space of a plane containing 
the axis $\R$, which shows that a tractrix is the image of a helix in $\Hy^2$ with a light-like axis under
the isometry between the hyperboloidal model of $\Hy^2$ and its half-space model.}
   
\emph{ We point out that surfaces in $\R^3$ whose unit normal vector field
makes a constant angle $\theta \in [0, \pi/2]$ with $\mathcal K$ have been investigated by Nistor and
Munteanu in \cite{mn}, who claimed that the only such surfaces 
were  rotation surfaces with $\R$ as axis (corresponding to $\theta=0$), 
open subsets of  planes containing $\R$ (corresponding to $\theta=\pi/2$),  the cylinder 
over a logarithmic spiral and Dini's helicoidal surface of constant negative curvature.
 In their proof of the classification of such surfaces, however, in the case $\theta\in (0, \pi/2)$
they choose local coordinates satisfying some properties, which turn out to exist only for Dini's surface
and the cylinder over a logarithmic spiral. As a consequence, their classification theorem
misses the remaining surfaces given by (\ref{eq:fCRkillingsup}) in terms of an arbitrary unit-speed curve
$\gamma\colon I\to \Hy^2\subset \R_1^3$ other than a horocycle and a helix in $\Hy^2$ with a light-like axis,
which give rise to the cylinder over  a logarithmic spiral and Dini's  surface, respectively.}
\end{remark}

\subsection{The principal direction property with respect to $\mathcal K$}

 The classification of submanifolds of $\R^{n+1}\setminus \R^{n-1}$ having the 
principal direction property with respect to $\mathcal K$ follows accordingly 
from part $(iv)$ of Lemma \ref{le:basic} and Theorem \ref{thm:PDmain}.

\begin{theorem}\label{thm:PDkilling} Let 
$f\colon N^{m-1}\times I\to \Hy^n\times \R\subset \R_1^{ n+2}$ be given by 
(\ref{eq:fPD}) in terms of a smooth regular curve 
$\gamma\colon I\to \Hy^k\times \R\subset \R_1^{k+2}$, 
$\gamma=(\gamma_1, \ldots, \gamma_{k+2})$,
such that $\gamma_{k+2}$ has nonvanishing derivative. 
Write $\pi\circ f\colon N^{m-1}\times I\to \Hy^n\subset \R^{n+1}_1$ as 
$\pi\circ f=\sum_{j=0}^n f_je_j$, where $\pi\colon \Hy^n\times \R\to \Hy^n$ is the projection
and $e_0, \ldots, e_{n}$ is a pseudo-orthonormal basis of $\R^{n+1}_1$ as in (\ref{eq:pseudo}).
 Then the restriction to the subset of regular points of the map 
 $\hat f\colon N^{m-1}\times I\to \R^{n+1}\setminus \R^{n-1}$ given by
\be\label{eq:fPDkilling}
\hat f(x,s) = \frac{1}{f_0(x,s)}(f_1(x,s), \ldots, f_{n-1}(x,s),\cos(\gamma_{k+2}(s)),
\sin(\gamma_{k+2}(s)))
\ee
is an immersion with the principal direction  property with 
respect to  $\mathcal K$.

Conversely, if $\hat f\colon M^m \to \R^{n+1} \setminus \R^{n-1}$, $m\geq 2$, is 
an isometric immersion that has the principal direction property with respect 
to $\mathcal K$, then  either it is a rotational submanifold having $\R^{n-1}$
as axis, or its image lies in a hyperplane that contains the subspace $\R^{n-1}$,
or it is locally given as  above 
(globally, if the integral curves of $\mathcal K_{\hat f}^T$ are defined on $\R$).
\end{theorem}

If, in the above statement, the smooth regular curve 
$\gamma\colon I\to \Hy^k\times \R\subset \R_1^{k+2}$, $\gamma=(\gamma_1, \ldots, \gamma_{k+2})$,
is such that $\bar\gamma\colon I\to \Hy^k\subset \R_1^{k+1}$, $\bar\gamma=(\gamma_1, \ldots, \gamma_{k+1})$,
has unit speed and 
$\gamma_{k+2}(s)=As$ for all $s\in I$, with $A\neq 0$, then it follows from 
Corollary \ref{cor:CRandPD} that the map $\hat f\colon M^m\to \R^{n+1}\setminus \R^{n-1}$ 
given by (\ref{eq:fPDkilling}) is an immersion with 
both the constant ratio and principal direction properties with 
respect to the Killing vector field $\mathcal K$, and that, conversely, 
any isometric immersion $\hat f\colon M^m \to \R^{n+1} \setminus \R^{n-1}$, $m\geq 2$, that
has both the constant ratio and principal direction properties with respect 
to $\mathcal K$  either is a rotational submanifold having $\R^{n-1}$
as axis, or its image lies in a hyperplane that contains the subspace $\R^{n-1}$,
or it is locally given in this way 
(globally, if the integral curves of $\mathcal K_{\hat f}^T$ are defined on $\R$). \vspace{1ex}

In the case of hypersurfaces, a more explicit description, similar to that in Corollary~\ref{cor:CRkillinghyp},
follows as before from Theorem \ref{thm:PDkilling} by taking into account Lemma \ref{le:element}.

\subsection{The case of conformal Killing vector fields}

Let $x_1, \ldots, x_{n+1}$ denote as before  the standard coordinates in $\R^{n+1}$ and 
$\d_{x_i}$  a unit vector field tangent to the $x_i$-coordinate curve, $1\leq i\leq n+1$.
The Lie algebra of \emph{conformal} Killing vector fields in $\R^{n+1}$ has dimension 
$\frac{1}{2}(n+2)(n+3)$
and is generated by the Killing vector fields given by (\ref{eq:ckvf}) and (\ref{eq:rkvf}),
by the radial vector field
$\displaystyle{\sum_{i=1}^{n+1}x_i \d_{x_i}}$,
and by the vector fields
\be\label{eq:ci}
{\cal C}_i=\frac{1}{2}(x_i^2-\sum_{j \neq i}x_j^2)\d_{x_i}+x_i\sum_{j \neq i}x_j\d_{x_j},\;\;1\leq i\leq n+1.
\ee

The isometric immersions into Euclidean space that have either the constant ratio or
the principal direction property with respect to any of such vector fields have been 
described in the previous sections, except for the vector fields ${\cal C}_i$ in (\ref{eq:ci}).

A description in the latter case follows from the observation that
\be \label{eq:ckvf2}
{\cal  I}_*\d_{x_i}=-\frac{2}{\sum_{i=1}^nx_i^2}\,{\cal C}_i, \;\;1\leq i\leq n+1,
 \ee
where ${\cal  I}$ is an inversion with respect to a unit sphere centered at the origin.
Indeed, in view of (\ref{eq:ckvf2}), it follows from part $(iii)$ (respectively, part $(iv)$)
of Lemma \ref{le:basic} and Theorem~\ref{thm:CRmain} (respectively, Theorem~\ref{thm:PDmain}),
with $\frac{\d}{\d t}=\d_{x_i}$, 
that any isometric immersion $f\colon M^m\to \R^n$ with the constant ratio 
(respectively, principal direction) property with respect to ${\cal C}_i$ is given by
$f={\cal  I}\circ \hat f$, where $\hat f\colon M^m\to \R^n$ is given as in 
Theorem~\ref{thm:CRmain} (respectively, Theorem~\ref{thm:PDmain}). 
In the hypersurface case, a more explicit description is given by such a composition with
$\hat f$ given as in Corollary $2$ (respectively, Theorem $1$) of \cite{to}.
Moreover, if $\hat f\colon M^m\to \R^n$ is as in Corollary \ref{cor:CRandPD},
with $\frac{\d}{\d t}=\d_{x_i}$, then
$f={\cal  I}\circ \hat f\colon M^m\to \R^n$ has both the constant ratio and
the principal direction properties with respect to ${\cal C}_i$.

\bigskip

\noindent Fernando Manfio and Ruy Tojeiro\\
Universidade de S\~ao Paulo\\
Instituto de Ci\^encias Matem\'aticas e de Computa\c c\~ao.\\
Av. Trabalhador S\~ao Carlense 400\\
13560-970 -- S\~ao Carlos\\
BRAZIL\\
\texttt{manfio@icmc.usp.br} and \texttt{tojeiro@icmc.usp.br}

\medskip

\noindent Joeri Van der Veken \\
KU Leuven, Department of Mathematics \\
Celestijnenlaan 200B -- Box 2400 \\
3001 Leuven\\
BELGIUM \\
\texttt{joeri.vanderveken@kuleuven.be}

\end{document}